\documentclass[a4paper,reqno,11pt]{amsart}

\usepackage[all]{xy}
\usepackage[english]{babel}
\usepackage{graphicx}
\usepackage{hhline}
\usepackage{xcolor}
\usepackage{wasysym}
\usepackage{amsmath,amsthm,amssymb}
\usepackage{verbatim}
\usepackage{geometry}
\usepackage{algorithm}
\usepackage{algorithmic}

\usepackage{hyperref}
\hypersetup{
    colorlinks,
    linkcolor={red!50!black},
    citecolor={blue!50!black},
    urlcolor={blue!80!black}
}

\usepackage{enumerate}

\usepackage{listings}

\definecolor{dkgreen}{rgb}{0,0.6,0}
\definecolor{gray}{rgb}{0.5,0.5,0.5}
\definecolor{mauve}{rgb}{0.58,0,0.82}

\newcommand{\p}{\mathbb{P}}

\newcommand{\N}{\mathbb{N}}
\newcommand{\Z}{\mathbb{Z}}
\newcommand{\Q}{\mathbb{Q}}
\newcommand{\C}{\mathbb{C}}
\newcommand{\R}{\mathbb{R}}

\newcommand{\vu}{\varnothing}

\newcommand{\T}{T}

\newcommand{\de}{\partial}

\newcommand{\gr}{\mathop{\rm Gr}\nolimits}
\newcommand{\oo}{\mathcal{O}}

\newcommand{\Hilb}{\mathop{\rm Hilb}\nolimits}
\newcommand{\Bl}{\mathop{\rm Bl}\nolimits}

\newcommand{\Sing}{\mathop{\rm Sing}\nolimits}
\newcommand{\codim}{\mathop{\rm codim}\nolimits}
\newcommand{\Sym}{\mathop{{\rm Sym}}\nolimits}
\newcommand{\Span}[1]{\langle#1\rangle}
\newcommand{\ov}[1]{\overline{#1}}

\newcommand{\reg}{R}
\newcommand{\irr}{\mathop{\rm Sing}\nolimits}
\newcommand{\SO}{\mathop{\rm SO}\nolimits}

\newcommand{\GL}{\mathop{\rm GL}\nolimits}

\newcommand{\Eig}{\mathop{\rm Eig}\nolimits}

\theoremstyle{plain}

\newtheorem{theorem}{Theorem}[section]
\newtheorem{proposition}[theorem]{Proposition}
\newtheorem{lemma}[theorem]{Lemma}   
\newtheorem{corollary}[theorem]{Corollary}

\theoremstyle{definition}

\newtheorem{example}[theorem]{Example}
\newtheorem{conjecture}[theorem]{Conjecture}
\newtheorem{thmx}{Theorem}
 
\newtheorem{definition}[theorem]{Definition}

\newtheorem{remark}[theorem]{Remark}

\title{Eigenschemes of Ternary Tensors}
\author
{Valentina Beorchia, Francesco Galuppi and Lorenzo Venturello}

\address
{Department of Mathematics and Geosciences,
  University of Trieste,
via Valerio 12/1, 34127 Trieste, Italy,
ORCID 0000-0003-3681-9045.}
\email{beorchia@units.it}

\address{Department of Mathematics and Geosciences, University of Trieste,
via Valerio 12/1, 34127 Trieste, Italy, ORCID 0000-0001-5630-5389.}
\email{fgaluppi@units.it}

\address{KTH Royal Institute of Technology, Lindstedtsv\"{a}gen 25, 11428 Stockholm, Sweden, ORCID 0000-0002-6797-5270.}
\email{lven@kth.se}


\begin{document}

	\maketitle

\begin{abstract}
We study projective schemes arising from eigenvectors of tensors, called eigenschemes. After some general results, we give a birational description of the variety parametrizing eigenschemes of general ternary symmetric tensors and we compute its dimension. Moreover, we characterize the locus of triples of homogeneous polynomials defining the eigenscheme of a ternary symmetric tensor. Our results allow us to implement algorithms to check whether a given set of points is the eigenscheme of a symmetric tensor, and to reconstruct the tensor. Finally, we give a geometric characterization of all reduced zero-dimensional eigenschemes. The techniques we use rely both on classical and modern complex projective algebraic geometry.
\end{abstract}

\section{Introduction}
Tensors are natural generalizations of matrices in higher dimension. Just as matrices are crucial in linear algebra, tensors play the same role in multilinear algebra, and find applications in several branches of mathematics, as well as many applied sciences. And, just as for matrices, it is possible to give a notion of eigenvectors and eigenvalues for tensors, as introduced independently in 2005 by Lim  \cite{Lim} and Qi  \cite{Qi}. 
As we shall soon see, when dealing with eigenvectors it is not restrictive to focus on partially symmetric tensors. In this paper, a partially symmetric tensor is an element of $\Sym^{d-1}\C^{n+1}\otimes\C^{n+1}$.

Let $T$ be a partially symmetric tensor. If we choose a basis for $\C^{n+1}$, then we can identify $T$ with a tuple of homogeneous polynomials $(g_0,\dots,g_n)$. An \emph{eigenvector} of $T$ is a vector $v$ such that

\begin{equation}\label{eq: definizione di autovettore}
(g_0(v), \ldots, g_n(v))=\lambda v\end{equation}for some constant $\lambda$. Such vectors are called \emph{E-eigenvectors} in \cite[Section 2.2.1]{QZ}. Since 
property \eqref{eq: definizione di autovettore} is preserved under scalar multiplication, it is natural to consider the rational map $T:\p^n\dashrightarrow\p^n$ defined by $T(P)=(g_0(P):\ldots:g_n(P))$ and regard eigenvectors as points in $\p^n$; hence the name \emph{eigenpoints} instead of eigenvectors. In the symmetric case we define the eigenpoints of a homogeneous polynomial $f$ as the fixed points of the polar map, or equivalently the eigenpoints of the partially symmetric tensor $\nabla f =(\de_0f,\dots,\de_n f)$.

Tensor eigenpoints appear naturally in optimization.  As an example, consider the problem of maximizing a polynomial function $f$ over the unit sphere in $\R^{n+1}$. Eigenvectors of the symmetric tensor $f$ are critical points of this optimization problem. Another interesting framework in which eigenvectors of symmetric tensors arise is the variational context: by Lim's Variational Principle \cite{Lim}, given a symmetric tensor $f$, the critical rank one symmetric tensors for $f$ are exactly of the form $v^d$, where $v$ is an eigenvector of $f$. 
This has applications in low-rank approximation of tensors (see \cite{OttSod}) as well as maximum likelihood estimation in algebraic statistics. Finally, in \cite{OO} Oeding and Ottaviani employ eigenvectors of tensors to produce an algorithm to compute Waring decompositions of homogeneous polynomials.

\begin{remark}
\label{rmk: scelta di una base ortonormale}\label{rmk: dipende dalla scelta di una base}
The property of being an eigenvector is invariant under the action of the group $\SO_{n+1}(\R)$ of orthogonal transformations, as proven in \cite[Theorem 2.20]{QZ}. However, it is important to stress that it is not invariant under linear transformations. For this reason, when we speak of \textquotedblleft the eigenscheme of a tensor $T$\textquotedblright we are abusing terminology. In order to avoid ambiguities, we choose once and for all a basis $\{v_0,\dots,v_n\}$ of $\C^{n+1}$ such that $v_0,\dots,v_n$ are real vectors and they are orthonormal with respect to the Euclidean scalar product in $\R^{n+1}$. With this choice we identify the space $\Sym^d\C^{n+1}$ of symmetric tensors with the space $\C[x_0,\dots,x_n]_d$ of degree $d$ homogeneous polynomials. In the same way we identify the space $\Sym^{d-1}\C^{n+1}\otimes\C^{n+1}$ with the space $(\Sym^{d-1}\C^{n+1})^{\oplus (n+1)}$. In order to remember this important choice, we always think of partially symmetric tensors as tuples of polynomials. Once the basis is chosen, we can safely compute eigenvectors as in \eqref{eq: definizione di autovettore}.
\end{remark}
Because of their many real-life applications, and since they are not linear invariant, one might argue that the right setting to study eigenvectors is affine Euclidean geometry over the real numbers. Nevertheless, we will study them through the lenses of complex projective geometry. One of the reasons is that being an eigenpoint of a tensor is an algebraic condition, described by the vanishing of minors of suitable matrices. 

\begin{definition}\label{def:eigenscheme}
Let $T= (g_0, \ldots, g_n)$ be a partially symmetric tensor. The \emph{eigenscheme} of $T$ is the closed subscheme $E(T)\subseteq\p^n$ defined by the $2 \times 2$ minors of
	\[\left(
	\begin{matrix}
	x_0 & x_1 & \dots & x_n \\
	g_0 & g_1 & \dots & g_n
	\end{matrix}\right).
	\]
	When $T$ is symmetric, that is, a homogeneous polynomial $f$, we have $g_i=\de_if$ and we  denote its eigenscheme by $E(f)$.
\end{definition}

\begin{remark}\label{non è restrittivo considerare i parzialmente simmetrici}
It is possible to define the eigenscheme of a tensor which is not partially symmetric, as in \cite[Definition 1.1]{CartSturm} or in \cite[Section 1]{ASS}. However, with the choice of a basis described in Remark \ref{rmk: dipende dalla scelta di una base}, it is apparent from the definition that for every tensor $T$ there exists a partially symmetric tensor $T^\prime$ such that $E(T^\prime)=E(T)$. For this reason, it is not restrictive to consider partially symmetric tensors. When we discuss properties of elements in $(\C^{n+1})^{\otimes d}$, it is understood that the basis is chosen.
\end{remark}

Eigenpoints have also connections to dynamical systems. If we call $T:\p^n\dashrightarrow\p^n$ the rational map defined by $T(P)=(g_0(P):\ldots:g_n(P))$, then $P$ is an eigenpoint if and only if either $P$ is a fixed point for $T$ or $T$ is not defined at $P$. In order to make this distinction precise, we will use the notion of regular eigenscheme, introduced in \cite[Definition 2.3]{CGKS} by \c{C}elik, Galuppi, Kulkarni and Sorea (see Definition \ref{def: regular eigenscheme}).

The basic questions about eigenschemes concern their dimension, their degree and, in the zero-dimensional case, their configuration. The case $d=2$ of matrices is discussed by Abo, Eklund, Kahle and Peterson in \cite{matrici}. Other partial results are proved by Abo, Seigal and Sturmfels in \cite{ASS}. The eigenscheme of a  general symmetric tensor of format $3\times 3\times\ldots\times 3$ is zero-dimensional, and its degree $d^2-d+1$ was first computed in \cite{Kuz} by Kuznetsov and Kohlshevnikov. 
It is particularly interesting to look at tensors with unusual or pathological properties, such as having eigenpoints with higher multiplicity, or a higher dimensional eigenscheme. This was addressed in \cite[Section 4]{ASS} and in \cite{Abo}.

For the convenience of the reader, we present here the structure of the paper and we summarize our main contributions. In Section \ref{sec: basic}, we use basic algebraic geometry to gain some knowledge on the general geometry of eigenschemes. For a general tensor, we give a sharp bound on the number of eigenpoints which can lie on a linear space (Proposition \ref{prop: configurations}) and we describe an infinite family of polynomials with a positive-dimensional regular eigenscheme (Proposition \ref{thm: eigenscheme 1-dim all degs}). Section \ref{sec: ternary forms} focuses on ternary forms. The following theorem collects the results from Proposition \ref{pro: se d è dispari l'autoschema generico identifica f} and Corollary \ref{cor: Eig is rational}.

\begin{thmx}
Let $f,g\in\C[x_0,x_1,x_2]_d$ be general homogeneous polynomials and let $q=x_0^2+x_1^2+x_2^2$ be the isotropic conic. Then $E(f)=E(g)$ if and only if $g=\lambda f+\mu q^k$ for some $\lambda \in\C^*$ and $\mu\in\C$. Moreover, the variety parametrizing zero-dimensional eigenschemes of ternary forms is rational and has dimension
\[\begin{cases}
  \binom{d+2}{2}-1 & \mbox{ if $d$ is odd}\\
  \binom{d+2}{2}-2 & \mbox{ if $d$ is even}.
\end{cases}\]
\end{thmx}
Here we find a discrepancy with \cite[Theorem 5.5]{ASS}. In Theorem \ref{thm: defining equations}, we characterize which tuples of generators arise from Definition \ref{def:eigenscheme}. The precise statement is the following.
\begin{thmx}
The subset of $(\C[x_0,x_1,x_2]_d)^{\oplus 3}$ parametrizing determinantal bases of ideals of eigenschemes of ternary forms is a linear subspace defined by
\begin{align*}
	    x_0f_0 - x_1 f_1 + x_2 f_2 = \partial_0 f_0 - \partial_1 f_1 + \partial_2 f_2 = 0.
\end{align*}
\end{thmx}
It turns out that the topic of eigenschemes is directly related to some beautiful classical results due to Bateman and Laguerre, and we explore this connection in Sections \ref{sec: bateman} and \ref{sec: laguerre}. 
Roughly speaking, the generators of the ideal of the eigenscheme of a ternary form define a rational map $\p^2\dashrightarrow\p^2$. By employing both classical and modern techniques, we are able to characterize the fibers of such maps and as a consequence we deduce the possible configurations
of eigenschemes in any degree.
\begin{thmx}
Let $d\ge 3$. A set $Z\subseteq\p^2$ of $d^2-d+1$ points is the eigenscheme of a tensor in $(\C^{3})^{\otimes d}$ if and only if
\begin{enumerate}
    \item $\dim I_Z(d)=3$,
    \item $Z$ contains no $d+1$ collinear points, and
    \item for every $k\in\{2,\dots,d-1\}$, no $kd$ points of $Z$ lie on a degree $k$ curve.
\end{enumerate}
\end{thmx}
This is shown in Theorems \ref{thm: facile} and \ref{teorema finale} and generalizes \cite[Proposition 2.1]{OS1} and \cite[Theorem 5.1]{ASS}. One of the ingredients is the numerical character of a set of points in $\p^2$, introduced by Gruson and Peskine.

In this work it was useful to make experiments with the software Macaulay2, freely available at \texttt{www.math.uiuc.edu/Macaulay2}.

\section{Basic results on eigenschemes}\label{sec: basic}

Since the introduction of eigenschemes of tensors is relatively recent, many questions are still unanswered. For instance, given a configuration of eigenpoints, it is natural to ask if they are in special position. Under mild assumptions, we give a sharp bound on the number of eigenpoints which can lie on a linear space.

While the general tensor has a zero-dimensional eigenscheme, it is very interesting to see what happens when it has a positive-dimensional component. The goal to classify all possible subschemes of $\p^n$ arising as eigenschemes of tensors seems to be still open. However, we have some initial results in this direction. We first recall the definition of {\it regular eigenscheme}.

\begin{definition}\label{def: regular eigenscheme} 
Let $T=(g_0,\dots,g_n)\in(\Sym^{d-1}\C^{n+1})^{\oplus(n+1)}$ be a partially symmetric tensor. The \emph{irregular eigenscheme} of $T$ is the subscheme $\irr(T)\subseteq\p^n$ defined by the ideal $(g_0,\dots,g_n)\subseteq\C[x_0,\dots,x_n]$. If $T$ is symmetric, then we identify it with a homogeneous polynomial $f$. In this case $g_i=\de_if$ and the irregular eigenscheme is indeed the singular locus of $f$. The residue of $E(T)$ with respect to $\irr(T)$ is called the \emph{regular eigenscheme} and denoted by $\reg(T)$. We can compute the ideal of $\reg(T)$ as the saturation
	\[
	I_{\reg(T)}=I_{\overline{\reg(T)}}= I_{E(T)}:(I_{\irr(T)})^{\infty}.
	\]
\end{definition}

In the present section, we bound the degree of a $(n-1)$-dimensional component of the eigenscheme and we give a necessary condition for the regular eigenscheme of a polynomial to have positive dimension. What strikes us as remarkable is the role played by the isotropic quadric.
We are able to employ it in order to exhibit a family of plane curves with a positive-dimensional regular eigenscheme.

We start by proving a few easy properties that will be useful in the rest of the article. When $f\in\C[x_0,\dots,x_n]_3$, several features of $E(f)$ have been worked out in \cite{CGKS}. We generalize some of them for arbitrary $d$. The next lemma addresses the number of eigenpoints of the general tensor.

\begin{lemma}\label{lem:nonempty}
Let $d\ge 3$ and let $\T\in(\C^{n+1})^{\otimes d}$.
	\begin{enumerate}
		\item \label{bullet: count general eigenpoints} If $\dim E(\T)=0$, then $\deg E(\T)=\frac{(d-1)^{n+1}-1}{d-2}$. If $\T$ is general, then $E(\T)$ is reduced.
\item \label{bullet: at least an eigenpoint} $E(T)\neq \vu$. In particular, if $f$ is a smooth polynomial then $R(f)\neq\emptyset$.
\end{enumerate}
\end{lemma}
\begin{proof}
Part \eqref{bullet: count general eigenpoints} was already solved in \cite[Theorem 2.1]{CartSturm}. For part \eqref{bullet: at least an eigenpoint}, we observe that the proof presented in \cite[Lemma 2.2(2)]{CGKS} for cubics applies verbatim to any $d\ge 3$ and also to partially symmetric tensors. The extension to general tensors comes from Remark \ref{non è restrittivo considerare i parzialmente simmetrici}.
\end{proof}

For cubic polynomials, \cite[Proposition 3.3]{CGKS} gives an inequality between the dimension of the singular locus and the dimension of the regular eigenscheme. Once more, it is immediate to see that the proof can be extended not only to every $d\ge 3$, but also to not necessarily symmetric tensors.

\begin{proposition} \label{pro: dimE and dimS} 
Let $d\ge 3$. If $T=(g_0,\ldots,g_n)\in(\Sym^{d-1}\C^{n+1})^{\oplus(n+1)}$, then $\dim \irr(T) + 1 \ge \dim \reg(T)$. In particular, $\dim\reg(f)=0$ whenever $f$ is a smooth polynomial.
\end{proposition}

\begin{proof}
By definition,  eigenvectors $v$ of $T$ are solutions of the equations
\begin{equation}\label{eq: naif}
    g_i(v)=\lambda v_i
\end{equation}
for every $i\in\{0,\dots,n\}$ and for some $\lambda\in\C$. Solutions with $\lambda=0$ are irregular eigenpoints, while solutions with $\lambda\neq 0$ are regular. In order to compute the equations of $E(T)$, one approach is to consider the ideal generated by \eqref{eq: naif} in the larger ring $\C[x_0,\dots,x_n,\lambda]$ and to eliminate $\lambda$. However, equations \eqref{eq: naif} are not homogeneous. In order to work with homogeneous forms, we raise $\lambda$ to the power $d-2$. This does not change the eigenvector. Consider then a projective space of dimension $n+1$, with coordinates $x_0,\dots,x_n,\lambda$. Let $H$ be the hyperplane of this $\p^{n+1}$  defined by $\lambda=0$ and let $\pi:\p^{n+1}\dashrightarrow \p^n$ be the projection from $(0:\ldots:0:1)$. Then $$E(T)=\pi(\{(x_0:\ldots:x_n:\lambda)\in\p^{n+1}\mid g_i(x_0:\ldots:x_n)=\lambda^{d-2} x_i\ \forall i
\})$$
and
$\Sing(T)=\pi(H\cap\{(x_0:\ldots:x_n:\lambda)\in\p^{n+1}\mid g_i(x_0:\ldots:x_n)=\lambda^{d-2} x_i\ \forall i
\})$.
The projection $\pi$ corresponds to the operation of eliminating $\lambda$.  Since all fibers of $\pi$ have dimension 1, we have
\begin{align*}
\dim \irr(T) &= \dim(\pi^{-1} \irr(T))-1 = \dim \left( \pi^{-1}E(T)\cap H \right) -1.\end{align*}
Now we observe that $H$ is not contained in $\pi^{-1}E(T)$, so intersecting $\pi^{-1}E(T)$ with $H$ decreases the dimension of at least 1, therefore
\begin{align*}
\dim \left( \pi^{-1}E(T)\cap H \right) \ge\dim(\pi^{-1}E(T))-2 = \dim E(T)-1\ge\dim \reg(T)-1.
\end{align*}If $f$ is a smooth polynomial, then this implies $\dim\reg(f)\le 0$. We already know that $\dim\reg(f)\neq-1$ by Lemma \ref{lem:nonempty}(\ref{bullet: at least an eigenpoint}).
\end{proof}

It is natural to wonder how basic operations between polynomials reflect on their eigenscheme. The next lemma gathers some elementary results in this direction.
\begin{lemma}
\label{lem: first properties} Let $T=(g_0,\dots,g_n)\in(\Sym^{d-1}\C^{n+1})^{\oplus(n+1)}$.
\begin{enumerate}
\item If $S\in(\Sym^{d-1}\C^{n+1})^{\oplus(n+1)}$, then $E(T+S)\supseteq E(T)\cap E(S)$\label{bullet: sum of two forms}.
\item \label{bullet: parzialmente simmetrici con lo stesso autoschema}Let $h\in\C[x_0,\dots,x_n]_{d-2}$ and $\lambda\in\C^*$. If $T^\prime=(\lambda g_0+x_0h,\dots,\lambda g_n+x_nh)$, then $E(T)=E(T^\prime)$.
\item \label{item: sommare la quadrica isotropa non cambia l'autoschema}Assume that $d=2k$ is even and that $T$ is symmetric, that is, $T$ corresponds to a form $f\in\C[x_0,\dots,x_n]_{2k}$. Let $q=x_0^2+\ldots+x_n^2$. Then $E(f)=E(\lambda f+\mu q^k)$ for every $\lambda\in\C^*$ and every $\mu\in\C$.
\end{enumerate}
\end{lemma}
\begin{proof} Part \eqref{bullet: sum of two forms} and \eqref{bullet: parzialmente simmetrici con lo stesso autoschema} follow from a direct computation. Part \eqref{item: sommare la quadrica isotropa non cambia l'autoschema} is a special case of part \eqref{bullet: parzialmente simmetrici con lo stesso autoschema}, where we take $h=2\mu kq^{k-1}$.
\end{proof}

We turn our attention to configurations of eigenpoints. Possible configurations of eigenpoints of cubic surfaces are addressed in \cite[Section 5]{CGKS}. The authors prove that they are parametrized by an open set of a linear subspace of the Grassmanian $\gr(3,\p^{14})$. Now we want to take a more basic approach and see how many eigenpoints can be in a linear special position. Since all linear spaces of a given dimension are equivalent up to orthogonal transformation, we can give them coordinates thanks to  \cite[Theorem 2.20]{QZ}.
We start with a technical result.

\begin{lemma}	\label{lem: restrizione ad un sottospazio lineare}
	Let $f\in\C[x_0,\dots,x_n]$ and let $L\subseteq\p^n$ be a linear subspace of equations $x_{k+1}=\ldots=x_n=0$. Denote by $f_{|L}\in\C[x_0,\dots,x_k]$ the restriction of $f$ to $L$, defined by
	\[
f_{|L}(x_0,\dots,x_k)=f(x_0,\dots,x_k,0,\dots,0)
	\]
and let $E(f_{|L})\subseteq L$ be the eigenscheme of this restriction. Then $L\cap E(f)\subseteq E(f_{|L})$.
\begin{proof} 
By Definition \ref{def:eigenscheme}, \begin{align*}
&I_{L\cap E(f)}=(x_i\de_jf-x_j\de_if	\mid 0\le
i<j\le n)+(x_{k+1},\dots,x_n)
\subseteq\C[x_0,\dots,x_n].
	\end{align*}
		On the other hand, $I_{E(f_{|L})}=(x_i\de_jf-x_j\de_if\mid 0\le i<j\le k)\subseteq\C[x_0,\dots,x_k]$, so when we regard it under the embedding of $L$ in $\p^n$, it becomes
		\[I_{E(f_{|L})}=(x_i\de_jf-x_j\de_if\mid 0\le i<j\le k)+(x_{k+1},\dots,x_n)\subseteq\C[x_0,\dots,x_n].\]
Therefore $I_{E(f_{|L})}\subseteq I_{L\cap E(f)}$ and this means that $E(f_{|L})\supseteq L\cap E(f)$.
\end{proof}
\end{lemma}
Observe that the analogous statement of Lemma \ref{lem: restrizione ad un sottospazio lineare} does not hold for the regular eigenscheme. Consider for instance the polynomial $q=x_0^2+\ldots+x_n^2$ defining the isotropic quadric $Q$. Then $E(q)=R(q)=\p^n$. Let $p\in Q$ and let $L=T_pQ$. Then $p\in R(f)$ but $p\in\Sing(f_{|L})$, hence $p\in R(f)\setminus R(f_{|L})$. Therefore in general $L\cap R(f)$ does not need to be a subscheme of $R(f_{|L})$.

\begin{proposition}\label{prop: configurations}
Let $T\in(\C^{n+1})^{\otimes d}$.
\begin{enumerate}
    \item If $\dim E(T)=0$, then $E(T)$ contains no $d+1$ collinear points.
\item Assume that $T$ is symmetric and let $f\in\C[x_0,\dots,x_n]_d$ be the corresponding homogeneus polynomial. Let $L\subseteq\p^n$ be any $k$-dimensional linear space. If $\dim(E(f_{|L}))=0$, then
$$\deg (L\cap R(f))\le\deg (L\cap E(f))\le \frac{(d-1)^{k+1}-1}{d-2}.$$
\end{enumerate}
\end{proposition}
\begin{proof}\begin{enumerate}
    \item Thanks to Remark \ref{non è restrittivo considerare i parzialmente simmetrici}, we can assume that $T$ is partially symmetric, hence $I_{E(T)}$ is defined by $\binom{n+1}{2}$ polynomials of degree $d$. This means that $E(T)$ is the base locus of a linear system of degree $d$ divisors of $\p^n$. If $d+1$ of the base points are collinear, then the whole line is contained in the base locus by B\'{e}zout's theorem, and therefore the base locus of $I_{E(T)}(d)$ has positive dimension.
\item By Lemma \ref{lem:nonempty}, both the eigenscheme and the regular eigenscheme of the restriction $f_{|L}\in\C[x_0,\dots,x_k]_d$ have degree at most $\frac{(d-1)^{k+1}-1}{d-2}$. The second inequality of our statement comes from Lemma \ref{lem: restrizione ad un sottospazio lineare}, while the first inequality follows from Definition \ref{def: regular eigenscheme}.
\end{enumerate}
\end{proof}

In particular, the general polynomial has at most $d$ collinear eigenpoints, $d^2-d+1$ coplanar eigenpoints, and so on.
Notice that the assumption that $\dim E(f_{|L})=0$ is necessary for our proof. Indeed, there are forms with a zero-dimensional eigenscheme whose restriction to a linear space has a positive-dimensional eigenscheme, as witnessed by the following example.

\begin{example}\label{example: restriction can change the dimension of the eigenscheme}Let $g=x_0^2x_1+x_1^2x_2+x_1x_3^2$ and let $L=V(x_4)\subseteq\p^4$. With the command \texttt{random} in Macaulay2, we take  a general element $h\in\C[x_0,\dots,x_4]_2$ and we set $f=g+x_4h$. Then $\dim E(f)=\dim R(f)=0$, while $\dim E(f_{|L})=\dim R(f_{|L})>0$.
\end{example}

Now we want to shed some light on eigenschemes of positive dimension. The locus of such tensors has not been studied yet. Our first result generalizes \cite[Lemma 38]{Maccioni} and will also play a role in Section \ref{sec: ternary forms}. For the special case $n=2$, observe that our statement already appears in proof of \cite[Thorem 5.5]{ASS}.

\begin{lemma}\label{lem: due G_i nulli}
Let $T= (g_0, \ldots, g_n)\in(\Sym^{d-1}\C^{n+1})^{\oplus(n+1)}$ and let
\[
\Lambda=\p\left(\langle x_ig_j-x_jg_i\mid 0\le i<j\le n
\rangle\right)\subseteq\p\left(\C[x_0,\dots,x_n]_d\right)
\]
be the linear subspace spanned by the generators of $I_{E(T)}$. Then $\dim \Lambda\le 0$ if and only if there exists $h\in\C[x_0,\dots,x_n]_{d-2}$ such that $g_i=x_ih$ for every $i\in\{0,\ldots,n\}$. In this case $\Lambda=\vu$.

When $T$ is symmetric, we can identify $T$ with a polynomial $f\in\C[x_0,\dots,x_n]_d$. In this case, $\dim \Lambda\le 0$ if and only if there exist $\lambda\in\C$ and $k\in\N$ such that\\ $f=\lambda(x_0^2 +\ldots+x_n^2)^k$. In particular, $d=2k$ is even.
\end{lemma}
\begin{proof} One implication is easy. Indeed, if $g_i=x_ih$ then it is straightforward to verify that $x_ig_j-x_jg_i=0$. Let us consider the other direction. By hypothesis, for every $0\le i<j\le n$ there exists $\lambda_{ij}\in\C$ such that
 \[x_i g_j-x_jg_i=\lambda_{ij}(x_0g_1-x_1g_0).\]
 For $i=0$ and $i=1$ we obtain that
 \[
\begin{cases}
 x_0 g_j-x_jg_0=\lambda_{0j}(x_0 g_1 -x_1 g_0 )\\ x_1 g_j -x_j g_1 =\lambda_{1j}(x_0 g_1 -x_1 g_0 )
\end{cases}\]
for every $j\in\{2,\dots,n\}$, and thus
\begin{equation}\label{eq: tutti multipli di G0, ipotesi}
\begin{cases}
 x_0( g_j -\lambda_{0j} g_1 )=(x_j-\lambda_{0j}x_1) g_0 \\ x_1( g_j +\lambda_{1j} g_0 )=(x_j+\lambda_{1j}x_0) g_1 .
\end{cases}
 \end{equation}
 Therefore there exist $h_0,h_1, h_{0j} , h_{1j} \in\C[x_0,\dots,x_n]_{d-2}$ such that
 \begin{equation}\label{eq: tutti multipli di G0, prima sostituzione}
 \begin{cases}
 g_0 =x_0h_0\\  g_1 =x_1h_1\\
 g_j -\lambda_{0j}g_1= h_{0j} (x_j-\lambda_{0j}x_1)\\
 g_j +\lambda_{1j} g_0 = h_{1j} (x_j+\lambda_{1j}x_0). \end{cases}
 \end{equation}
 We substitute \eqref{eq: tutti multipli di G0, prima sostituzione} in \eqref{eq: tutti multipli di G0, ipotesi} to obtain
 \begin{equation*}
     \begin{cases}
 x_0 h_{0j} (x_j-\lambda_{0j}x_1)=x_0h_0(x_j-\lambda_{0j}x_1)\\ x_1 h_{1j} (x_j+\lambda_{1j}x_0)=x_1h_1(x_j+\lambda_{1j}x_0).
\end{cases}\Rightarrow
\begin{cases}
  h_{0j} =h_0\\  h_{1j} =h_1
\end{cases}
\mbox{ for every } j\ge 2.
 \end{equation*}
If we substitute $ h_{0j} =h_0$ and $ h_{1j} =h_1$ in \eqref{eq: tutti multipli di G0, prima sostituzione} and we subtract, then we get
  \begin{equation*}
 \begin{cases}
 g_j -\lambda_{0j}x_1h_1=h_0(x_j-\lambda_{0j}x_1)\\
 g_j +\lambda_{1j}x_0h_0=h_1(x_j+\lambda_{1j}x_0)     \end{cases}\Rightarrow
h_1(x_j-\lambda_{0j}x_1+\lambda_{1j}x_0)=h_0(x_j-\lambda_{0j}x_1+\lambda_{1j}x_0)
\end{equation*}
and so $ h_0=h_1$. Define $h=h_0=h_1=h_{0j}= h_{1j} $. Then \eqref{eq: tutti multipli di G0, prima sostituzione} implies
\begin{equation*}
 g_j-\lambda_{0j}x_1h=h(x_j-\lambda_{0j}x_1)\Rightarrow g_j=x_jh\mbox{ for every } j\in\{0,\dots,n\}.
\end{equation*}
We conclude by observing that in this case $I_{E(T)}=0$, so $\Lambda=\vu$. Note that the case $\dim\Lambda=0$ never occurs.

Now that we have proven our claim about partially symmetric tensors, we assume that $T$ is symmetric. This means that there is a polynomial $f\in\C[x_0,\dots,x_n]_d$ such that $g_j=\de_jf$ for every $j\in\{0,\dots,n\}$. We argue by induction on $d$. The cases $d\le 1$ are straightforward, so we assume that $d\ge 2$. By Schwarz' theorem we have
\begin{align*}
    x_i\de_jh=\de_j(x_ih)=\de_j g_i=\de_j\de_if=\de_i\de_jf=\de_ig_j=\de_i(x_jh)=x_j\de_ih,
\end{align*}
so $x_i\de_jh=x_j\de_i h$ for every $i<j$. This implies that 
\[
    \dim(\p(\langle x_i\de_j h - x_j\de_i h \mid 0\leq i < j \leq n\rangle))\leq 0.
\]
Since $\deg h< d$, we deduce by induction hypothesis that
$h=\lambda^\prime(x_0^2 +\ldots+x_n^2)^{k^\prime}$ for some $\lambda^\prime\in\C$ and some $k^\prime\in\N$. By Euler's formula,
\[\deg(f)\cdot f
=\sum_{j=0}^n x_j\de_jf=\sum_{j=0}^n x_j^2h=\lambda^\prime(x_0^2+\ldots+x_n^2)^{k^\prime +1}.
\]
If we set $k=k^\prime+1$ and $\lambda=\frac{\lambda^\prime}{\deg(f)}$, then our statement on $f$ is proven.
 \end{proof}

\begin{corollary}\label{cor: no pure dimensional hypersurface}
Let $T\in(\C^{n+1})^{\otimes d}$.
\begin{enumerate}
\item $E(T)$ is not a pure dimensional hypersurface. 
\item \label{item: degree of component} The $(n-1)$-dimensional components of $E(T)$ have degree at most $d-1$.
\item\label{item: auteschema è l'intero spazio} Assume that $T= (g_0, \ldots, g_n)\in(\Sym^{d-1}\C^{n+1})^{\oplus (n+1)}$. Then $E(T)=\p^n$ if and only if there exists $h\in\C[x_0,\dots,x_n]_{d-2}$ such that $g_i=x_ih$ for every $i\in\{0,\ldots,n\}$.
\end{enumerate}
\end{corollary}
\begin{proof}Thanks to Remark \ref{non è restrittivo considerare i parzialmente simmetrici}, we can assume that $T$ is partially symmetric.
\begin{enumerate}
        \item Suppose that $E(T)$ is a pure dimensional hypersurface. Then $I_{E(T)}$ is principal and so $\dim\Lambda=0$, in contradiction to Lemma \ref{lem: due G_i nulli}.
\item A $(n-1)$-dimensional component $V$ of $E(T)$ is a fixed component of the linear system $\Lambda\subseteq\p(\C[x_0,\dots,x_n]_d)$, so it has degree at most $d$. If we had $\deg V=d$, then all generators would be multiples of $V$, so $\dim\Lambda=0$, in contradiction to Lemma \ref{lem: due G_i nulli}.
    \item One implication follows from Lemma \ref{lem: first properties}\eqref{bullet: parzialmente simmetrici con lo stesso autoschema}, while the other is Lemma \ref{lem: due G_i nulli}.
\end{enumerate}\end{proof}

We know that the general polynomial has a zero-dimensional eigenscheme. The locus of homogeneous polynomials with a positive-dimensional eigenscheme is not well understood. This set contains all polynomials with a positive-dimensional singular locus, but there exist also forms $f$ for which $\dim R(f)> 0$.    We take a more geometric viewpoint and we determine a necessary condition for this to happen. This condition involves the position of $f$ relative to the isotropic quadric.

\begin{proposition}\label{pro: relation eigenpoints - tangent space to the isotroquadric}
Let $f\in\C[x_0,\dots,x_n]$ be a homogeneous polynomial and let $q=x_0^2+\ldots+x_n^2$ be the isotropic quadric. 
\begin{enumerate}
\item \label{bullet: When do eigenpoints belong to $V(f)$?} Let $P\in\p^n$. Then $P\in V(f)\cap R(f)$ if and only if $P\in V(q) \cap V(f)$ and $T_P V(q)=T_P V(f)$.
    \item If $V(f)$ meets $V(q)$ transversally, then $\dim R(f)\le 0$.
\end{enumerate}
\begin{proof} 
\begin{enumerate}
\item Assume that $P\in R(f)\cap V(f)$. Since $P\in T_P V(f)$, the condition $\nabla f(P)\cdot P=0$ is satisfied.  The assumption $P\in R(f)$ implies that $P=\nabla f(P)$, thus $P\cdot P=\nabla f(P)\cdot P=0$ and therefore $P\in V(q)$. If we let $x=(x_0,\dots,x_n)$ be the vector of the variables, then the space $T_P V(q)$ is defined by the linear polynomial $P\cdot x$, which is the same as the polynomial $\nabla f(P) \cdot x$, because $P\in R(f)$. Hence $T_P V(q)=T_P V(f)$. Conversely, assume that $T_P V(q)=T_P V(f)$. Up to a multiplicative constant, this means $P\cdot x=\nabla f(P)\cdot x$. Therefore $P=\nabla f(P)$ and so $P\in R(f)$. Now we prove that $P\in V(f)$. Since $P\in V(q)\cap R(f)$, the  Euler identity implies that
$\deg(f)\cdot f(P)=\nabla f(P)\cdot P=P\cdot P=0$.
\item Assume by contradiction that $\dim R(f)\ge 1$. Then $\overline{R(f)}\cap V(f)\neq\vu$. A point in $\overline{R(f)}\cap V(q)$ cannot belong to $\Sing(f)$, because $V(q)$ and $V(f)$ are transverse, hence $R(f)\cap V(f)\neq\vu$. By part \eqref{bullet: When do eigenpoints belong to $V(f)$?}, there exists a point $P\in R(f)\cap V(f)\cap V(q)$ such that $T_P V(f)=T_P V(q)$. This contradicts our hypothesis that $V(q)$ and $V(f)$ are transverse.
\end{enumerate}
\end{proof}
\end{proposition}

Now that we have a necessary condition for the eigenscheme to have positive dimension, we would like to find a sufficient one. 
Of course, it is easy to exhibit examples with a positive-dimensional singular locus, so we focus on $R(f)$. The following results allow us to build an infinite family of curves in $\p^2$ with a zero-dimensional singular locus and a positive-dimensional regular eigenscheme. As a byproduct, this will also show that the bound we established in Corollary \ref{cor: no pure dimensional hypersurface}\eqref{item: degree of component} is sharp for ternary forms.
\begin{lemma}\label{lem: C tangent}
Let $c$ be a plane conic tangent to the isotropic conic at $P$ and $Q$. Let $l=\langle P,Q\rangle$. Then there exists $\lambda\in\C$ such that $x_i\partial_j c -x_j\partial_i c = \lambda l(x_i\partial_j l-x_j\partial_i l)$ for every $0\leq i<j\leq 2$. In particular, $E(c)$ consists of the line $l$ and the unique point in $E(l)$.
\end{lemma}
\begin{proof}
Thanks to \cite[Theorem 2.20]{QZ}, we may assume that $\de_il\neq 0$ for every $i\in\{0,1,2\}$. First we want to prove that each of the conics $x_i\partial_j c -x_j\partial_i c$ is the union of $l$ and another line. For this purpose we show that if we evaluate $x_i\partial_j c -x_j\partial_i c$ at a general point of $l$ we get 0. A point of $l=\Span{P,Q}$ can be written as $aP+bQ$ for some $a,b\in\C$. Let $(P_0:P_1:P_2)$ and $(Q_0:Q_1;Q_2)$ be coordinates for $P$ and $Q$. Since $c$ is tangent to $q$ at $P$, we have that 
$\de_ic(P)=\de_iq(P)=2P_i$. The same holds for $Q$, hence
\begin{align*}	(x_i\partial_j c-x_j\partial_i c)&(a P + b Q) = (a P_i +b Q_i)\partial_jc(a P + b Q) -  (a P_j +b Q_j)\partial_ic(a P + b Q)\\
	& = (a P_i +b Q_i)(a\partial_jc(P)+b\partial_jc(Q)) - (a P_j +b Q_j)(a\partial_ic(P)+b\partial_ic(Q))\\
	& = 2(a P_i +b Q_i)(a P_j +b Q_j) - 2(a P_j+b Q_j)(a P_i +b Q_i) = 0.
	\end{align*}
If we identify $l$ with the polynomial defining it, then $l\mid x_i\partial_j c-x_j\partial_i c$, so there exist linear forms $s_{ij}$ such that $x_i\partial_j c-x_j\partial_i c = l s_{ij}$. Next we determine $s_{ij}$. By equating the coefficients of $x_i^2$ and $x_j^2$ of the last expression we obtain $\partial_i\partial_jc = (\partial_il)(\partial_is_{ij})$ and $-\partial_i\partial_jc = (\partial_jl)(\partial_js_{ij})$. If we call $\lambda_{ij}=-\frac{\de_i\de_jc}{(\de_il)(\de_jl)}\in\C$, this implies that $\partial_i s_{ij} = \lambda_{ij}\partial_j l$ and $\partial_j s_{ij} = \lambda_{ij}\partial_i l$. Let $k$ be the third index and $A$ be the point with coordinates $x_i=x_j=0$. Then
$$0=(x_i\partial_j c-x_j\partial_i c)(A)=l(A)s_{ij}(A)=\de_kl\cdot\de_ks_{ij}.$$
Thanks to our assumption that $\de_kl\neq 0$, we deduce that $\partial_k s_{ij} =0$ and therefore $s_{ij}=x_i\de_is_{ij}+x_j\de_js_{ij}=\lambda_{ij}(x_i\partial_jl-x_j\partial_il)$ by Euler's identity. Thus
	\begin{equation*}\label{eq: 1-dim}
	x_i\partial_j c-x_j\partial_i c = \lambda_{ij} l(x_i\partial_j l-x_j\partial_i l).
	\end{equation*}
Now we prove that $\lambda_{01}=\lambda_{02}=\lambda_{12}$. Indeed
	\begin{align*}
	0 &= x_0(x_1\partial_2 c - x_2 \partial_1 c) - x_1(x_0\partial_2 c - x_2 \partial_0 c) +x_2(x_0\partial_1 c - x_1 \partial_0 c) \\
	&=x_0\lambda_{12}l(x_1\partial_2 l - x_2 \partial_1 l) - x_1\lambda_{02}l(x_0\partial_2 l - x_2 \partial_0 l) +x_2\lambda_{01}l(x_0\partial_1 l - x_1 \partial_0 l)\\
	&= l((\lambda_{12}-\lambda_{02})x_0x_1\partial_2l +(-\lambda_{12}+\lambda_{01})x_0x_2\partial_1l+(\lambda_{02}-\lambda_{01})x_1x_2\partial_0l)
	\end{align*}
	This implies that 	\begin{equation}\label{eq: k are the same}
	(\lambda_{12}-\lambda_{02})x_0x_1\partial_2l +(-\lambda_{12}+\lambda_{01})x_0x_2\partial_1l + (\lambda_{02}-\lambda_{01})x_1x_2\partial_0l
	\end{equation}
	is the zero polynomial. Since all partial derivatives of $l$ are not zero, we get the desired equality by evaluating \eqref{eq: k are the same} at $(1:0:0)$, $(0:1:0)$ and $(0:0:1)$. 
	\end{proof}
\begin{proposition}\label{thm: eigenscheme 1-dim all degs}
	Let $P,Q$ be points on the isotropic conic $q=x_0^2+x_1^2+x_2^2$ and let $l$ be the line $\langle P,Q\rangle$. Let $c_1,\dots,c_s$ be conics which are all tangent to $q$ at $P$ and $Q$. Then	\begin{enumerate}
\item\label{item: coniche} $E(c_1\cdot c_2\cdot\ldots\cdot c_s		)$  is the scheme-theoretic union of $l$, a curve of degree $2(s-1)$ and the point $E(l)$.
\item\label{item: coniche pi retta} $E(l\cdot c_1\cdot c_2\cdot\ldots\cdot c_s	)$  is the scheme-theoretic union of a curve of degree $2s$ and the point $E(l)$.
	\end{enumerate}
Such 1-dimensional components of the eigenscheme are also tangent to $q$ at $P$ and $Q$.
\end{proposition}	
\begin{proof}\begin{enumerate}
\item	Let $f = c_1\cdots c_s$. By Lemma \ref{lem: C tangent} we obtain 
	\[x_i\partial_j f - x_j \partial_i f 
	= \sum_{u=1}^s (x_i\partial_j c_u-x_j\partial_i c_u)\prod_{v\neq u}c_v
	= l(x_i\partial_j l-x_j\partial_i l)\sum_{u=1}^s \left(\lambda_u\prod_{v\neq u}c_v\right)
	\]
	for some $\lambda_u\in \C$. It is important to observe that in Lemma \ref{lem: C tangent} we proved that the numbers $\lambda_u$ do not depend on $i$ and $j$. Hence the common zero locus of the three polynomials $x_i\partial_j f - x_j \partial_i f$ are $l$, the degree $2(s-1)$ curve $\sum_{u=1}^s \lambda_u\prod_{v\neq u}c_v$ and the unique point satisfying $x_i\partial_j l-x_j\partial_i l=0$ for every $0\leq i<j\leq 2$. By definition, this point is $E(l)$.
\item Let $f = l\cdot c_1\cdots c_s$. By Lemma \ref{lem: C tangent} we obtain 
	\begin{align*}
x_i\partial_j f - x_j \partial_i f &= (x_i\partial_j l-x_j\partial_i l)\prod_{v=1}^sc_v+l\sum_{u=1}^s \left((x_i\partial_j c_u-x_j\partial_i c_u)\prod_{v\neq u}c_v\right)\\
&= (x_i\partial_j l-x_j\partial_i l)\left(\prod_{v=1}^sc_v +l^2\sum_{u=1}^s \left( \lambda_u\prod_{v\neq u}c_v\right)\right).
	\end{align*}
Hence $E(f)$ consists of the degree $2s$ curve $\prod_{v=1}^sc_v +l^2\sum_{u=1}^s k_u\prod_{v\neq u}c_v$ and the point $E(l)$.
	\end{enumerate}
\end{proof}

\begin{example}
	Let $P = (1:i:0)$ and $Q=(3i:-4i:5)$. Then $l$ is the line defined by $5x_0+5ix_1-(4+3i)x_2=0$ and $c$ is the conic $$(1-i)x_0^2+(1+i)x_1^2+\left(\frac{49}{25}-\frac{7}{25}i\right)x_2^2+2x_0x_1-\left(\frac{6}{5}-\frac{8}{5}i\right)x_0x_2-\left(\frac{8}{5}+\frac{6}{5}i\right)x_1x_2=0.$$
Theorem \ref{thm: eigenscheme 1-dim all degs} states that $E(l\cdot c)$ consists of the point $E(l) = \{(4-3i:3+4i:-5)\}$ together with the conic $c+\lambda l^2$. According to the proof of Lemma \ref{lem: C tangent}, we have
$$\lambda=\frac{\partial_0\partial_1c}{\partial_0l\cdot \partial_1l} = -\frac{2}{25}i.$$
\end{example}

\section{Eigenschemes of ternary forms}\label{sec: ternary forms}

When dealing with ternary forms, we are able to say something more precise on eigenschemes and their ideals. For instance, the ideal of the eigenscheme of a general plane curve is saturated and we can compute its Hilbert series. We also study the variety parametrizing configurations of eigenpoints, we compute its dimension and we construct a simple birational model. Furthermore, we characterize bases of ideals of eigenschemes of plane curves.

\begin{proposition}\label{pro: saturated} Let $T=(g_0,g_1,g_2)\in(\Sym^{d-1}\C^{3})^{\oplus 3}$ and $S=\C[x_0,x_1,x_2]$. If $\dim E(T)=0$, then $I_{E(T)}$ is saturated and the sequence
\begin{equation*}\label{eq: resolution}
0\to S(-d-1)\oplus 
S(-2d+1) \xrightarrow
{\left(
\begin{array}{cc}
x_0 & g_0\\
x_1& g_1\\
x_2&g_2\\
\end{array}
\right)
} S (-d)^{\oplus 3}\to I_{E(T)} \to 0
\end{equation*}
is a minimal free resolution of $I_{E(T)}$. As a consequence, the Hilbert series of $E(T)$ is $$
\Hilb(S/I_{E(T)},t)=\frac{1-3t^d+t^{d+1}+t^{2d-1}}{(1-t)^3}
$$
and $\deg E(T) = d^2-d+1$.
\end{proposition}
\begin{proof} Thanks to \cite[Section 20.4]{EisenCommut}, the quotient $S/I_{E(T)}$ is Cohen-Macaulay and therefore $I_{E(T)}$ is saturated. The minimal free resolution of $I_{E(T)}$ is presented in \cite[Section 5]{ASS} for the case in which $E(T)$ is reduced, and the argument carries over to any form with a zero-dimensional eigenscheme. 
\end{proof}
\begin{remark}\label{rem: saturate}
 Let $d\ge 3$ and let $T\in(\C^{n+1})^{\otimes d}$. If $\dim E(T)=0$, then $I_{E(T)}$ is still saturated (see for instance \cite[Corollary A2.13]{EisenCommut}). However, the ideal does not need to be saturated, in general. As an example, if $f = x_0x_2^2+x_1x_3^2+x_2x_3x_4$ then $I_{E(f)}$ is not saturated. In this case $\dim E(f)=2$.
\end{remark}

Next we take a step back and we look at all possible configurations of eigenpoints in the plane. We start by defining the variety which parametrizes them.

\begin{definition}
\label{def: phi}    Let $d\in\N$ and let $(\p^2)^{(d^2-d+1)}$ be the symmetric power of $\p^2$ with itself $d^2-d+1$ times, which parametrizes unordered sets of $d^2-d+1$ points in the plane. Define
\[
\begin{matrix}
\phi_d:&\p(\C[x_0,x_1,x_2]_d)&\dashrightarrow & (\p^2)^{(d^2-d+1)}\\
&f&\mapsto &E(f).
\end{matrix}
\]
This rational map is defined on all forms $f$ such that $E(f)$ is reduced and zero-dimensional. Following \cite[Section 5]{ASS}, we define $\Eig_{d,\Sym}$ to be the closure of the image of $\phi_d$.
\end{definition}

\begin{remark}\label{rmk: dimension of  Eig}
In the proof of \cite[Theorem 5.5]{ASS}, the authors show that $\Eig_{d,\Sym}$ is birational to the projectivization of the quotient vector space $U/H$, where
\[
 U=\left\lbrace\begin{pmatrix} x_0 & x_1 & x_2 \\
 \de_0f & \de_1 f & \de_2 f
 \end{pmatrix} \mid f\in \C[x_0,x_1,x_2]_d
 \right\rbrace
\]
and 
\[
 H=\left\lbrace \begin{pmatrix}
    1 & 0 \\
    g & \lambda 
 \end{pmatrix} \mid g\in \C[x_0,x_1,x_2]_{d-2}\mbox{, } \begin{pmatrix}
 1 & 0 \\
 g & \lambda
 \end{pmatrix}U\subseteq U\mbox{ and }
 \lambda\in \C\right\rbrace.
\]
They claim that $\dim H=1$ and deduce that  $\dim\Eig_{d,\Sym}=\dim\p(\C[x_0,x_1,x_2]_d)$. In particular, they state that $\phi_{d}$ is generically finite. However, this is true only when $d$ is odd. Indeed, Lemma \ref{lem: first properties}(\ref{item: sommare la quadrica isotropa non cambia l'autoschema}) shows that the general fiber of $\phi_{d}$ has dimension at least one whenever $d$ is even. In the next lemma we  show that $\dim H = 2$ when $d$ is even, thereby computing the generic fiber of $\phi_{d}$
.
\begin{lemma}\label{lem: dim of H}
    Let $q=x_0^2+x_1^2+x_2^2$ and let $d\ge 2$. If $d$ is odd, then 
    \[
    H = \left\lbrace \begin{pmatrix}
    1 & 0 \\
    0 & \lambda 
 \end{pmatrix} \mid 
 \lambda\in \C\right\rbrace.
    \]
    If $d=2(k-1)$ is even, then
    \[
    H = \left\lbrace \begin{pmatrix}
    1 & 0 \\
    g & \lambda 
 \end{pmatrix} \mid g=\mu q^{k-2}\mbox{ and }
 \lambda,\mu\in \C\right\rbrace.
    \]
 In particular, if $d$ is odd then $\dim H =1$ and if $d$ is even then $\dim H =2$.
\begin{proof}
Let $g\in \C[x_0,x_1,x_2]_{d-2}$ such that
\[\begin{pmatrix}
    1 & 0 \\
    g & \lambda 
 \end{pmatrix}\in H.\]
 Then for every $f\in \C[x_0,x_1,x_2]_{d}$ there exists $F\in \C[x_0,x_1,x_2]_{d}$ such that
 \[
 \begin{pmatrix}
    1 & 0 \\
    g & \lambda 
 \end{pmatrix}\begin{pmatrix} x_0 & x_1 & x_2 \\
 \de_0f & \de_1 f & \de_2 f
 \end{pmatrix}=
 \begin{pmatrix} x_0 & x_1 & x_2 \\
 \de_0F & \de_1 F & \de_2 F
 \end{pmatrix},
 \]
 hence $\de_i F=x_i g+\lambda\de_i f$ for every $i\in\{0,1,2\}$. This implies that
$$F=\frac1d(x_0\de_0F+x_1\de_1 F+x_2\de_2F)=\frac{qg}{d}+\lambda f$$
 and therefore $x_ig+\lambda \de_i f=\de_i F=\frac{2x_ig+q\de_ig}{d}+\lambda \de_i f$. Thus \begin{equation}
     \label{eq: Bernd}
(d-2)x_ig=q\de_ig.
 \end{equation} If $d=2$, then $g$ is constant, as required. If $d=3$, then $x_ig=q\de_ig$ and so $g=0$, because $q$ is irreducible and $g$ is a linear form. If $d\ge 4$, then there exists $g_1\in\C[x_0,x_1,x_2]_{d-4}$ such that $g=g_1q$. Therefore
 \begin{align*}
(d-2)x_ig_1q=q\de_i(g_1q)&\Rightarrow (d-2)x_ig_1=q\de_ig_1+2x_ig_1\Rightarrow (d-4)x_ig_1=q\de_ig_1.
\end{align*}
This means that $g_1$ satisfies the same property that $g$ satisfies by \eqref{eq: Bernd}, so we can repeat the procedure and say that there exists $g_2\in\C[x_0,x_1,x_2]_{d-6}$ such that $g_1=g_2q$. At each step, the degree of the form decreases by 2, thus preserving the parity of the degree. There are two possibilities for $g_{k-2}$. If $d=2k-1$ is odd, then $g_{k-2}$ is a linear form and relation \eqref{eq: Bernd} implies $g_{k-2}=0$ and so $g=0$. On the other hand, if $d=2k-2$ is even, then $g_{k-2}$ is a constant, so $g=qg_1=q^2g_2=\ldots=q^{k-2}g_{k-2}$.
\end{proof}
\end{lemma}

We want to study the variety $\Eig_{d,\Sym}$. We start by computing its dimension.
\begin{proposition}\label{pro: se d è dispari l'autoschema generico identifica f}\label{cor: dimension of Eig}
Let $f,g\in\C[x_0,x_1,x_2]_d$ be general homogeneous forms and let $q=x_0^2+x_1^2+x_2^2$ be the isotropic conic.
\begin{enumerate}
    \item If $d=2k$ is even, then $E(f)=E(g)$ if and only if $g=\lambda f+\mu q^k$ for some $\lambda\in\C^*$ and $\mu\in\C$. In particular, the general fiber of the map $\phi_d$ of Definition \ref{def: phi} is a line and $\dim\Eig_{d,\Sym}=\binom{d+2}{2}-2$.
\item \label{item: ident odd}If $d$ is odd, then $E(f)=E(g)$ if and only if $g=\lambda f$ for some $\lambda\in\C^*$. In particular, $\phi_d$ is birational and $\dim\Eig_{d,\Sym}=\binom{d+2}{2}-1$.
\end{enumerate}
\end{proposition}
\begin{proof}
One implication by Lemma \ref{lem: first properties}\eqref{item: sommare la quadrica isotropa non cambia l'autoschema}. Conversely, assume that $E(f)=E(g)$
. Since $f$ and $g$ are general, \cite[Theorem 2.1]{ASS} implies that both $E(f)$ and $E(f+g)$ are reduced of dimension zero and they have the same degree.  By Lemma \ref{lem: first properties},
\[E(f)=E(f)\cap E(g)
\subseteq E(f+ g),\]
so $E(f)=E(f+ g)$.
Assume that $d$ is odd. By Lemma \ref{lem: dim of H}, $E(f)$ has only one preimage under $\phi_d$, up to scalar. The only possibility is that $f$ is a multiple of $f+g$. On the other hand, if $d=2k$ is even, then the general fiber of $\phi_d$ is birational to $H$, as we pointed out in Remark \ref{rmk: dimension of  Eig}. We conclude by Lemma \ref{lem: dim of H}.
\end{proof}
\end{remark}

\begin{remark}\label{def di Td}
Proposition \ref{pro: se d è dispari l'autoschema generico identifica f} implies that $\Eig_{d,\Sym}$ is a rational variety whenever $d$ is odd, and we now prove that this also holds when $d$ is even. The map $\phi_d$ can be factored. Let $\alpha$ be the projectivization of the linear map
\[\begin{matrix}
\C[x_0,x_1,x_2]_d &\to &(\C[x_0,x_1,x_2]_d)^{\oplus 3}\\
f&\mapsto &(x_i\de_j f-x_j\de_i f)_{i<j},
\end{matrix}\]
and call 
\begin{equation}\label{Td per esteso}
T_d= \overline {\alpha (\p(\C[x_0,x_1,x_2]_d))} \subseteq\p((\C[x_0,x_1,x_2]_d)^{\oplus 3})
\end{equation}
the closure of the image of $\alpha$. Then $T_d$ is a linear subspace of  $\p((\C[x_0,x_1,x_2]_d)^{\oplus 3})$ and there is a commutative diagram
\begin{displaymath}
    \xymatrix{
 \p(\C[x_0,x_1,x_2]_d)\ar@{-->}[rr]^{\phi_d}\ar[rd]_{\alpha} & & \Eig_{d,\Sym}\\
         & T_d\ar@{-->}[ru]_{\psi_d} &
}
\end{displaymath}
where $\psi_d(f_0,f_1,f_2)$ is the set of points defined by the ideal $(f_0,f_1,f_2)$. Proposition \ref{cor: dimension of Eig} allows us to compute the general fiber of $\psi_d$ as well.
\end{remark}

\begin{corollary}\label{cor: Eig is rational}
Let $d\in\N$. 
 Then $\psi_d: T_d\dashrightarrow\Eig_{d,\Sym}$ is birational.
\begin{proof}
Assume first that $d$ is odd. By Lemma \ref{lem: due G_i nulli}, $\alpha$ is an isomorphism. Since $\phi_d=\psi_d\circ\alpha$ is birational by Proposition \ref{pro: se d è dispari l'autoschema generico identifica f}\eqref{item: ident odd}, $\psi_d$ is birational as well.
Now assume that $d$ is even. The fibers of $\alpha$ are lines by Lemma \ref{lem: due G_i nulli}, and the general fiber of $\phi_d$ is a line by Proposition \ref{pro: se d è dispari l'autoschema generico identifica f}. This implies that $\psi_d$ is generically injective.
\end{proof}
\end{corollary}

Since every triple $(f_0,f_1,f_2)\in T_d$ arises as minors of a certain matrix, we will refer to the elements $f_0$, $f_1$ and $f_2$ as \emph{determinantal} equations of an eigenscheme. The linear space $T_d$ plays an important role in our setting. It parametrizes all determinantal bases of ideals of eigenschemes and it is a birational model for $\Eig_{d,\Sym}$. For this reason, we seek to understand it better. We start by stating a more general result on the defining equations of the eigenscheme of a partially symmetric tensor. This lemma and its proof are precisely the equivalence (i)$\Leftrightarrow$(iii) of \cite[Theorem 7.3.13]{Dolgachev}, which we will recall in its complete statement in Section \ref{sec: laguerre}. The reason why we state this equivalence separately is that it does not require the hypothesis of irreducibility which appears in \cite[Theorem 7.3.13]{Dolgachev}, and we will take advantage of this fact. Next result is related to \cite[Proposition 5.2]{ASS}.

\begin{lemma}\label{lem: i implica iii}
Let $f_0,f_1,f_2\in\C[x_0,x_1,x_2]_d$. Then 
\begin{equation}\label{relazione lineare}
x_0f_0+x_1f_1+x_2f_2=0
\end{equation}
 if and only if there exist $g_0,g_1,g_2\in\C[x_0,x_1,x_2]_{d-1}$ such that
\begin{equation}\label{eq: minors lemma}
    f_0 = x_1g_2-x_2g_1,\qquad f_1 = x_2g_0-x_0g_2,\qquad 
    f_2 = x_0g_1-x_1g_0.
\end{equation}
\end{lemma}

\begin{proof}
If $f_0,f_1,f_2$ satisfy \eqref{eq: minors lemma}, then it is immediate to check that they satisfy \eqref{relazione lineare} as well. Conversely, assume that \eqref{relazione lineare} holds and let $S = \C[x_0,x_1,x_2]$.
The Koszul complex in the ring
$S$ is an exact sequence of $S$-modules
$$
0\to S\xrightarrow{\alpha} S^{\oplus 3} \xrightarrow{\beta} S^{\oplus 3} \xrightarrow{\gamma} S \to S/(x_0,x_1,x_2)
\to 0,
$$
where the maps are  $\alpha(h)=(h x_0,h x_1,hx_2)$, $\gamma (h_0,h_1,h_2)=h_0 x_0 + h_1 x_1 + h_2 x_2$ and $\beta$ is defined by the matrix
$$
\left(
\begin{array}{ccc}
0 & -x_2 & x_1\\
x_2 & 0 & -x_0\\
-x_1 & x_0 & 0 \\
\end{array}
\right).
$$
The syzygy $x_0f_0+x_1f_1+x_2f_2=0$ implies that $(f_0, f_1, f_2)$ is in the kernel of $\gamma$, and since the Koszul complex is exact, the triple $(f_0,f_1, f_2)$ lies in the image of $\beta$. It follows that there exist $g_0,g_1,g_2\in\C[x_0,x_1,x_2]_{d-1}$ such that \eqref{eq: minors lemma} holds.
\end{proof}

\begin{remark}\label{rmk: scelta del segno per f1}
We point out that in Lemma \ref{lem: i implica iii} we follow the choice of the signs of $f_0,f_1,f_2$ given in \cite[Theorem 7.3.13]{Dolgachev}. We shall see in Section 5 that such a choice is convenient when investigating a geometric characterization of zero-dimensional reduced eigenschemes. The opposite sign for $f_1$ gives rise to the following equivalent formulation of Lemma \ref{lem: i implica iii}:

Let $f_0,f_1,f_2\in\C[x_0,x_1,x_2]_d$.
 Then $
x_0f_0-x_1f_1+x_2f_2=0
$ if and only if there exist $g_0,g_1,g_2\in\C[x_0,x_1,x_2]_{d-1}$ such that 
$
    f_0 = x_1g_2-x_2g_1$, $f_1 = x_0g_2-x_2g_0$ and $f_2 = x_0g_1-x_1g_0
$.
\end{remark}

The choice of sign for $f_1$ illustrated in Remark \ref{rmk: scelta del segno per f1} will be functional to next result. We give necessary and sufficient conditions for a triple of polynomials to belong to the locus $T_d$ of determinantal triples of generators for symmetric tensors.

\begin{theorem}\label{thm: defining equations}Let $d\ge 2$. Three homogeneous forms $f_0,f_1,f_2\in \C[x_0,x_1,x_2]_d$ are the determinantal equations defining the eigenscheme of a form $f\in\C[x_0,x_1,x_2]_d$ if and only if 
	\begin{align}
	    &x_0f_0 - x_1 f_1 + x_2 f_2 = 0\label{eq: second}\mbox{ and }\\
	&\partial_0 f_0 - \partial_1 f_1 + \partial_2 f_2 = 0 \label{eq: first}.
	\end{align}
\end{theorem}

\begin{proof}
It is immediate to check that conditions \eqref{eq: second} and \eqref{eq: first} are necessary. On the other hand, if they are satisfied then Remark \ref{rmk: scelta del segno per f1} implies that there exist $g_0,g_1,g_2\in \C[x_0,x_1,x_2]_{d-1}$ such that 
	\begin{equation}\label{eq: g and h}
		f_0 = x_1g_2-x_2g_1, \qquad f_1 = x_0g_2-x_2g_0, \qquad f_2 = x_0g_1-x_1g_0.
	\end{equation}
We will prove that there exists $f\in\C[x_0,x_1,x_2]_d$ such that $x_ig_j-x_jg_i=x_i\de_jf-x_j\de_if$ for every $0\leq i<j\leq 2$. If we substitute \eqref{eq: g and h} in \eqref{eq: first} we obtain
	\begin{align}
	0 & =  \partial_0 f_0 - \partial_1 f_1 + \partial_2 f_2 \nonumber\\
	& = x_1\partial_0 g_2  - x_2\de_0g_1 - x_0\de_1g_2+x_2\de_1g_0+x_0\de_2g_1-x_1\de_2g_0\nonumber\\
	&  =  -(x_0(\de_1g_2-\de_2g_1) - x_1(\de_0 g_2-\de_2g_0) + x_2(\de_0g_1-\de_1g_0))\label{eq: induction}.
	\end{align}
We proceed by induction on $d$. If $d=2$, then $\de_i g_j-\de_j g_i$ are constants and so \eqref{eq: induction} implies that $\de_i g_j-\de_j g_i=0$ for every $i<j$. But $(\de_1g_2-\de_2g_1,\de_0 g_2-\de_2g_0,\de_0g_1-\de_1g_0)$ is the curl of the vector field $(g_0,-g_1,g_2)$, defined everywhere in $\C^3$. Hence the field is conservative, which means that there exists $f\in \C[x_0,x_1,x_2]_2$ such that $g_i=\de_if$ for every $i\in\{0,1,2\}$.

If $d>2$, then \eqref{eq: induction} does not immediately imply that $\de_i g_j-\de_j g_i=0$. However, \eqref{eq: induction} shows that the triple $(\de_1g_2-\de_2g_1,\de_0 g_2-\de_2g_0,\de_0g_1-\de_1g_0)$ satisfies \eqref{eq: second}. It is easy to see that it also satisfies \eqref{eq: first}, so by induction hypothesis we know there exists $h\in \C[x_0,x_1,x_2]_{d-2}$ such that 
	\begin{align}\label{eq: curl}
		\de_1g_2-\de_2g_1 - x_1\de_2h +x_2\de_1h&=0,\nonumber\\
		\de_0 g_2-\de_2g_0 - x_0\de_2h + x_2\de_0h&=0, \\	\de_0g_1-\de_1g_0 - x_0\de_1h+x_1\de_0h&=0.\nonumber
	\end{align}
Observe that $(g_0,g_1,g_2)$ is not the unique triple in $\C[x_0,x_1,x_2]_{d-1}$ satisfying \eqref{eq: second}. If we let $k$ be the third index, then the triple $(g_0+x_0h,g_1+x_1h,g_2+x_2h)$ satisfies \begin{equation}\label{eq: quella con g_k}
	x_i(g_j+x_jh)-x_j(g_i+x_ih)=x_ig_j-x_jg_i=f_k.
	\end{equation}
	Again we observe that the left hand sides of equations \eqref{eq: curl} are precisely the curl of the vector field $(g_0+x_0h,-g_1-x_1h,g_2+x_2h)$. Hence
	there exists a polynomial $f\in\C[x_0,x_1,x_2]_d$, such that
	\begin{equation}\label{come trovare gradiente}
		\de_0 f = g_0+x_0h,\qquad \de_1 f = g_1+x_1h, \qquad \de_2 f = g_2+x_2h.
	\end{equation}
By substituting these expressions in the left hand side of \eqref{eq: quella con g_k} we deduce that
\[
	f_0 = x_1\de_2 f -x_2\de_1f, 
\qquad f_1=x_0\de_2f-x_2\de_0f, 
\qquad f_2=x_0\de_1f- x_1\de_0f.
	\]
\end{proof}

\begin{remark}
    We point out that Theorem \ref{thm: defining equations} can be also proved by means of computer algebra. Indeed, if we consider the Weyl algebra $\C[x_0,x_1,x_2]\langle \de_0,\de_1,\de_2\rangle$ and the map
    \[
    \varphi: D^{\oplus 3} \xrightarrow
        {\left(
        \begin{array}{ccc}
       x_0 & -x_1 & x_2 \\
        \partial_0 & -\partial_1 & \partial_2 \\
        \end{array}
        \right)
        }
     D^{\oplus 2},
    \]
then the statement of Theorem \ref{thm: defining equations} means that the kernel of $\varphi$ equals the image of the map $D\to D^{\oplus 3}$ given by $f\to (x_1\de_2f-x_2\de_1f,x_0\de_2f-x_2\de_0f,x_0\de_1f-x_1\de_0f)$. One way to do so is by using the package $D$-modules of the software Macaulay2 \cite{Dmodules}.
\end{remark}

The choice of sign we underline in Remark \ref{rmk: scelta del segno per f1} and we employ in Theorem \ref{thm: defining equations} appears to be the right one in order to generalize equations \eqref{eq: second} and \eqref{eq: first} to eigenschemes of symmetric tensors in $\C[x_0,\dots,x_n]_d$. Indeed, we formulate the following conjecture.

\begin{conjecture}
Let $n\geq 2$. A set $\{f_{ij}\mid0\leq i<j\leq n\}\subseteq \C[x_0,\dots,x_n]_d$ of $\binom{n+1}{2}$ homogeneous polynomials is the set of determinantal equations of the eigenscheme $E(f)$ of some $f\in\C[x_0,\dots,x_n]_d$ if and only if 
	\begin{align*}
x_if_{jk} - x_j f_{ik} + x_k f_{ij} = 0	    \mbox{ and }
\partial_i f_{jk} - \partial_j f_{ik} + \partial_k f_{ij} = 0 
	\end{align*}
	for every $0\leq i<j<k\leq n$.
\end{conjecture}

\begin{remark}
We can regard equalities \eqref{eq: second} and \eqref{eq: first} as linear conditions on the coefficients of $f_0$, $f_1$ and $f_2$. In other words, the two equations determine the linear equations defining $T_d$ in $\p((\C[x_0,x_1,x_2]_d)^{\oplus 3})$. Let us analyse them. Equation \eqref{eq: second} is the condition on a degree $d+1$ polynomial to be identically zero, so it corresponds to imposing $\binom{d+3}{2}$ linear conditions on the coefficients of $f_0$, $f_1$ and $f_2$. Equation \eqref{eq: first} imposes $\binom{d+1}{2}$ more conditions. If $d$ is even, then $$\dim T_d=\binom{d+2}{2}-2=\dim\p((\C[x_0,x_1,x_2]_d)^{\oplus 3})-\binom{d+3}{2}-\binom{d+1}{2},$$
so \eqref{eq: second} and \eqref{eq: first} give independent conditions. On the other hand, when $d$ is odd we have
$$\dim T_d=\binom{d+2}{2}-1=\dim\p((\C[x_0,x_1,x_2]_d)^{\oplus 3})-\binom{d+3}{2}-\binom{d+1}{2}+1,$$
so \eqref{eq: second} and \eqref{eq: first} give one condition less than expected. Let us observe that in the case $d=3$ some equivalent equations are given in \cite[Proposition 5.3]{ASS}.
\end{remark}

\begin{remark}
As observed in \cite[Section 5]{ASS}, Lemma \ref{lem: i implica iii} can be translated into an algorithm testing whether a given configuration of $d^2-d+1$ points in $\p^2$ is the eigenscheme of some tensor $T$, and reconstructing $T$ from its eigenpoints. We follow Remark \ref{rmk: scelta del segno per f1} for the choice of the sign of $f_1$. Given a set of points $Z\subseteq\p^2$, we proceed as follows.
\begin{itemize}
\item Determine a minimal set of generators of $I_Z$; if $I_Z$ is not generated by three polynomials of degree $d$, then $Z$ is not an eigenscheme.
\item If $I_Z$ has three generators $h_0,h_1,h_2$ of degree $d$, we look at the Hilbert-Burch matrix; if it has not the form
    ${\scriptsize\begin{pmatrix}
	l_0 & p_0 \\
	l_1 & p_1 \\
	l_2 & p_2
	\end{pmatrix}}$ with $l_0,l_1,l_2$
    linearly independent linear forms, then $Z$ is not an eigenscheme by \cite[Proposition 5.2]{ASS}.
\item If $l_0,l_1,l_2$ are linearly independent, set $A\in \GL_3(\C)$ to be the matrix such that 
$(l_0,l_1,l_2)^\top=A\cdot(x_0,x_1,x_2)^\top$.

If we set $(f_0,f_1,f_2)^\top=A\cdot (h_0,h_1,h_2)^\top$,
then $f_0,f_1,f_2$ satisfy equation \eqref{eq: second} by construction, and therefore $Z$ is an eigenscheme.
\item A partially symmetric tensor $T$ with $E(T)=Z$ is given by 
the triple of forms  $(g_0,g_1,g_2)^\top=A^{-1}\cdot (p_0.p_1,p_2)^\top$.
\end{itemize}
If such algoritm gives a positive answer, then $Z$ is the eigenscheme of a partially symmetric tensor and $f_0,f_1,f_2$ are its determinantal equations. Then we can employ Theorem \ref{thm: defining equations} to check if there exists a polynomial $f$ such that $Z=E(f)$. By Corollary \ref{cor: Eig is rational}, the determinantal equations of the eigenscheme of a polynomial are unique. We proceed as follows. 
\begin{itemize}
    \item 
Check if $f_0,f_1,f_2$ satisfy
also
\eqref{eq: first}; if not, then $Z$ is not the eigenscheme of a symmetric tensor.
\item If $f_0,f_1,f_2$ satisfy
\eqref{eq: first}, then there exists a polynomial $h\in \C [x_0,x_1,x_2]_{d-2}$ such that the equations \eqref{come trovare gradiente} hold. We conclude by solving the affine linear system \eqref{come trovare gradiente} in the coefficients of $f$ and $h$. If $d$ is odd, the solution is unique up to a proportionality factor. If $d$ is even, then there are infinite many possibility for $h$. Two non proportional solutions give rise to two polynomials differing by a multiple of a power of the isotropic conic.
\end{itemize}
\end{remark}

In the last part of this section we briefly schematise the two algorithms. As computing the minimal resolution of an ideal requires exact arithmetic, in the effective implementation we work over $\Q$. 
\begin{algorithm}[H]
	\textbf{Input:} $d\geq 2$, a set $Z$ of $d^2-d+1$ rational points in $\mathbb{P}^2$.\\
	$I =\bigcap_{p\in Z} I_p$, with $I_p\subseteq \mathbb{Q}[x_0,x_1,x_2]= S$ the homogeneous prime ideal defining $p$. \\
	\algorithmicif{ $I$ is not generated by three forms of degree $d$}:
	\text{\hspace{20pt}}\\
	\text{\hspace{20pt}}\algorithmicreturn{ False}\\
	\textbf{else} compute the minimal resolution of $I$: 
	\[
	0\to S(-d-1)\oplus S(-2d+1)\overset{{\scriptsize\begin{pmatrix}
	l_0 & p_0 \\
	l_1 & p_1 \\
	l_2 & p_2
	\end{pmatrix}}}{\longrightarrow} S(-d) \to I \to 0
	\]
	\text{\hspace{20pt}}\algorithmicif{ there exists $A\in\GL_3(\mathbb{Q})$: $(l_0,l_1,l_2)^\top= A\cdot (x_0,x_1,x_2)^\top$: \\ \text{\hspace{20pt}}\text{\hspace{20pt}}\algorithmicreturn{  $\left( \text{True}, A^{-1}\cdot(p_0,p_1,p_2)^\top\right)$}}\\
	\text{\hspace{20pt}}\textbf{else }\\
	\text{\hspace{20pt}}\text{\hspace{20pt}}\algorithmicreturn{   False }
	\vspace{5pt}
	\caption{Test if points are the eigenscheme of a tensor in $(\Sym^{d-1}\Q^{3})^{\oplus 3 }$}
	\label{alg: algorithm}
\end{algorithm}

The second algorithm tests if a configuration of points is in particular the eigenscheme of a symmetric tensor $f\in \mathbb{Q}[x_0,x_1,x_2]$. If so, it returns such a polynomial $f$. Observe that if $d$ is even, Proposition \ref{pro: se d è dispari l'autoschema generico identifica f} established that $f$ is only unique up to summing the $\frac{d}{2}$-th power of the isotropic quadric.

\begin{algorithm}[H]
	\textbf{Input:} $d\geq 2$, a set $Z$ of $d^2-d+1$ rational points in $\mathbb{P}^2$.\\
	\algorithmicif{ Algorithm \ref{alg: algorithm} returns False:}\\ 
	\text{\hspace{20pt}}\algorithmicreturn{ False} \\
	\textbf{else} let $(g_0,g_1,g_2)\in (\Sym^{d-1}\Q^{3})^{\oplus 3 }$ be the output of Algorithm \ref{alg: algorithm}.\\
	\text{\hspace{20pt}} $f_0= x_1g_2-x_2g_1$, $f_1=x_0g_2-x_2g_1$, $f_2=x_0g_1-x_1g_0$.\\
	\text{\hspace{20pt}}\algorithmicif{ $\partial_0 f_0 - \partial_1 f_1+\partial_2 f_2 \neq 0$:}\\
	\text{\hspace{20pt}}\text{\hspace{20pt}}\algorithmicreturn{ False}\\
	\text{\hspace{20pt}}\textbf{else} let $f\in \mathbb{Q}[x_0,x_1,x_2]_d$ and $h\in\mathbb{Q}[x_0,x_1,x_2]$ with unknown coefficients.\\
	\text{\hspace{20pt}}\text{\hspace{20pt}}Solve the affine linear system \eqref{come trovare gradiente} in the coefficients of $f$ and $h$.\\
	\text{\hspace{20pt}}\text{\hspace{20pt}}\algorithmicreturn{ $\left( \text{True}, f\right)$}
	\caption{Test if points are the eigenscheme of a tensor in $\mathbb{Q}[x_0,x_1,x_2]_d$}
	\label{alg: algorithm2}
\end{algorithm}

We have effectively implemented these two algorithms in Macaulay2. The code and a few examples of its usage can be found in the following repository:

\noindent \href{https://github.com/LorenzoVenturello/Geometry-of-ternary-tensor-eigenschemes}{https://github.com/LorenzoVenturello/Geometry-of-ternary-tensor-eigenschemes}.

\section{Eigenschemes as Bateman configurations}\label{sec: bateman}

In this section we will focus on the case $(n,d)=(2,3)$
. By Lemma \ref{lem:nonempty}, we know that there are $7$ eigenpoints. In \cite[Theorem 5.1]{ASS}, the authors prove that a configuration of seven points in $\p^2$ is the eigenscheme of a $3\times 3\times 3$ tensor if and only if no six of the seven points lie on a conic. 
As a consequence, the general set of 7 points in $\p^2$ is the eigenscheme of a tensor.

For symmetric tensors, this is no longer true. By Proposition \ref{pro: se d è dispari l'autoschema generico identifica f}, the eigenvariety $\Eig_{3,\Sym}\subseteq(\p^2)^{(7)}$ has dimension 9, so the general set of 7 points in $\p^2$ is not the eigenscheme of a plane cubic curve. Therefore, it is legitimate to wonder whether eigenpoints of ternary cubics are in special position. 
In many examples, we noticed that the eigenscheme of a cubic curve contains three collinear points. While investigating this phenomenon, we realized that there are remarkable connections between eigenschemes and several classical topics. The aim of this section is to underline these links, as well as to prove that the eigenpoints of the general plane cubic are in general position.

Configurations of $7$ points in $\p^2$ have always attracted the interest of algebraic geometers, see for instance \cite[Section 6.3.3]{Dolgachev}. Among them, there is an important class of configurations, first defined in \cite{Bateman} and discussed in \cite[Section 9]{OS1}.
\begin{definition}
Let $f\in\C[x_0,x_1,x_2]_3$ and let $g\in\C[x_0,x_1,x_2]_2$ be a smooth conic. We set $Z(g,f)\subseteq\p^2$ to be the subscheme defined by the minors of
    \[\begin{pmatrix}
    \de_0g & \de_1 g & \de_2g \\
    \de_0f & \de_1f & \de_2f
    \end{pmatrix}.\]
When $f$ is general, $Z(g,f)$ consists of $7$ reduced points, as shown in \cite[Lemma 9.1]{OS1}. In this case we call $Z(g,f)$ the \emph{Bateman configuration} associated to $g$ and $f$.
\end{definition}
If we consider the isotropic quadric $q=x_0^2+x_1^2+x_2^2$, then $Z(q,f)=E(f)$, so the eigenscheme of a general ternary cubic is a Bateman configuration. Thanks to \cite[Lemma 9.1]{OS1}, we recover the fact that no six of the seven points are contained in a conic. Whenever we have a set $Z\subseteq\p^2$ with these properties, it is possible to define a rational map associated to $Z$.

\begin{definition}
    Let $Z\subseteq\p^2$ be a set of seven points. If no six points of $Z$ are contained on a conic, then it is easy to see that the complex vector space $I_Z(3)$ has dimension $3$ and
    the base locus of the linear system $ \p(I_Z(3))$ is exactly $Z$. In this case, the \emph{Geiser map} associated to $Z$ is the rational map $\gamma_Z:\p^2\dashrightarrow\p^2$ defined by the linear system $\p(I_Z(3))$. By blowing-up the plane $\p^2$ along $Z$ we get a generically finite morphism 
    $$
    \widetilde \gamma _Z : \Bl_Z \p^2 \to \p^2.
    $$
\end{definition}

Geiser maps are a classical topic and several of their properties are understood. As an example, $\gamma_Z$ is generically finite of degree $2$. The ramification locus of $\widetilde \gamma _Z$ is given by the {\it Jacobian locus} $\Sigma$ defined by the determinant of the Jacobian $\text{Jac}(I_Z)$, that is the locus of singular points of the net (see \cite[Book I, Chapter IX, Theorem 25] {Cool}).

When $Z$ is general, $\Sigma$ is a curve of degree $6$ which is singular at $Z$ as illustrated in \cite[Book I, Chapter IX, Theorem 27] {Cool}. We define $B(Z)$ to be the branch locus of $\widetilde\gamma_Z$, that is the direct image of $\Sigma$. For modern references, see for instance \cite[Section 8.7.2]{Dolgachev} and \cite[Section 7]{OS1}, where it is proven that a general Geiser map is branched along a smooth L\"uroth quartic.

\begin{lemma}    \label{lem: branched is smooth iff no collinear} Let $Z\subseteq\p^2$ be a set of seven distinct points such that
no six points lie on a conic, and let $B(Z)\subseteq\p^2$ be the associated branch locus. If $B(Z)$ is a smooth curve
of degree four, then $Z$ contains no three points on a line.\end{lemma}

\begin{proof}
For a subset $Y\subseteq Z$ of six points, we denote by $\pi:\Bl_{Y}\p^2\dashrightarrow \p^2$ the projection from the seventh point. There are two possible cases.

Assume that there is a subset $Y$ of six points of $Z$, not containing any collinear triple. Since they do not lie on a conic, $\Bl_{Y}\p^2$ is 
isomorphic to a smooth cubic surface $S$ in $\p^3$. 
It is known (see \cite[Section 3]{OS2}) that the ramification curve for the projection $\pi$ is a quartic L\"uroth curve, which is singular if and only if the seventh point lies on one of the $27$ lines of $S$.
Now recall that such lines correspond to the six exceptional divisors of $\Bl_Y\p^2$, the six conics passing through five points of $Y$, and the $15$ lines joining two points
of $Y$. The exceptional divisors are excluded in our case, because the seven points are distinct, and the conics are excluded by hypothesis. Therefore $B(Z)$ is singular if and only if the seventh point is collinear with two other points.

Assume now that for any choice of six points of $Z$ there is always a collinear triple; then it is simple to check that we have at least three alignments. The blow-up of $\p^2$ in six points of $Z$ is isomorphic
to a singular irreducible cubic surface $S$ in $\p^3$. As before, denote by $\pi: S \dasharrow \p^2$ the projection from the image $A$ of the seventh point. The ramification curve $R_\pi$ is given by the points of tangency of the tangent lines to $S$ passing through $A$. The latter are the intersection of $S$ with the first polar $P_A$ of $S$ with respect to $A$, hence $R_\pi$ has degree $6$. Since $P_A$ intersects
$S$ tangentially in $A$, the curve $R_\pi$ is singular in $A$. Moreover, since any first polar of a hypersurface contains its singular locus, the ramification curve $R_\pi$ has at least one singular point distinct from $A$ by construction. It may happen that $S$ contains the line through $A$ and some singular point. In any case,
the image of
$R_\pi$ under the projection $\pi$ is either a singular curve of degree $4$, or a plane curve of degree $3$ or less, or a finite set points, which contradicts the hypothesis on $B(Z)$.
\end{proof}

\begin{remark}
In \cite[Proposition 7.1]{OS1}, the authors state that $B(Z)$ is smooth whenever $Z$ contains no six points on a conic. This is true only with the further assumption that no three points of $Z$ are collinear.
\end{remark}

\begin{proposition}
\label{pro: plane cubics, general position} If $f\in\C[x_0,x_1,x_2]_3$ is general, then $E(f)$ contains no 3 points on a line and no 6 points on a conic.
\end{proposition}
\begin{proof}
We already observed that $E(f)$ does not contain 6 points on a conic by \cite[Theorem 5.1]{ASS} or \cite[Lemma 9.1]{OS1}. Moreover, by the proof of \cite[Theorem 10.4]{OS1}, since $f$ is general the branch locus of the associated Geiser map is a smooth L\"uroth quartic. Hence by Lemma \ref{lem: branched is smooth iff no collinear} the eigenscheme $E(f)$
contains no collinear triples.
\end{proof}

\begin{example}\label{example: maple1}
The eigenscheme of the smooth plane cubic
\[f=x_0x_2^2+x_0^2x_2-2x_0x_1x_2+x_0^3+x_0^2x_1-x_0x_1^2-x_1^3\]
has no collinear triples. In this case the Jacobian curve has equation 
\scriptsize{
\begin{align*}
& 12x_0^6-18x_0^5x_1-210x_0^4x_1^2+30x_0^3x_1^3+96x_0^2x_1^4-54x_0x_1^5+6x_1^6-36x_0^5x_2-6x_0^4x_1x_2+6x_0^3x_1^2x_2\\
    &+294x_0^2x_1^3x_2+78x_0x_1^4x_2-6x_1^5x_2+66x_0^3x_1x_2^2+234x_0^2x_1^2x_2^2+156x_0x_1^3x_2^2-150x_1^4x_2^2+18x_0^3x_2^3\\
    &+114x_0^2x_1x_2^3+270x_0x_1^2x_2^3-36x_1^3x_2^3+66x_0^2x_2^4-6x_0x_1x_2^4+60x_1^2x_2^4-6x_0x_2^5+12x_1x_2^5-6x_2^6.
\end{align*}}
\normalsize
The branch curve $B(f)$ is the smooth quartic given by the equation 
\scriptsize{
\begin{align*}
x_0^4-7x_0^3x_1+5x_0^2x_1^2+4x_0x_1^3+x_1^4-x_0^3x_2+9x_0x_1^2x_2+3x_1^3x_2+14x_0x_1x_2^2-6x_1^2x_2^2-7x_1x_2^3+2x_2^4.
\end{align*}}
\normalsize
Moreover, by making experiments with random forms, it is possible to check that if $f$ is a general triangle, then $E(f)$ consists of 7 reduced points, no three of which on a line and no six of which on a conic. The same holds when $f$ is the union of a general conic and a line.
 \end{example}

\section{Eigenschemes as base loci of Laguerre nets}\label{sec: laguerre} 
In this section we analyze the
geometry of reduced zero-dimensional eigenschemes of tensors $T\in(\C^3)^{\otimes d}$ with $d\ge 4$. As usual, by Remark \ref{non è restrittivo considerare i parzialmente simmetrici} we assume that $T=(g_0,g_1,g_2)\in (\Sym^{d-1}(\C^3))^{\oplus 3}$.
We recall from Definition \ref{def:eigenscheme} that the ideal of the eigenscheme is generated by the $2\times 2$ minors of the matrix
\begin{equation}\label{matrice_di_definizione}
\left(
	\begin{matrix}
	x_0 & x_1 & x_2 \\
	g_0 & g_1 & g_2\\
	\end{matrix}\right).
\end{equation}
where $g_i \in \Sym^{d-1}(\C^3)$. We set
\begin{equation}\label{generatori determinantali}
f_0 =x_1g_2 - x_2g_1, \quad f_1 =x_2g_0-  x_0g_2,\quad f_2 =x_0g_1 - x_1g_0.
\end{equation}
It turns out that such eigenschemes are a class of configurations of points arising as base loci of the classical {\it 
Laguerre nets}, see \cite[Book II, Chapter IV, Section 3]{Cool}.

\begin{definition}
Let $d\ge 4$ and let $V\subseteq \C[x_0,x_1,x_2]_d$ be a vector subspace of dimension $3$. Let $\Lambda =\p(V)\subseteq \p ( \C[x_0,x_1,x_2]_d)$ be a net of plane curves of degree $d$. If there exists a basis $f_0, f_1, f_2$ of $V$ such that
\begin{equation}\label{Laguerre di Coolidge}
x_0f_0 + x_1f_1 + x_2f_2 = 0,
\end{equation}
then $\Lambda$ is called a {\it 
Laguerre net}.
\end{definition}

\begin{remark}
By choosing as
basis of $I_{E(T)}$ the minors of \eqref{matrice_di_definizione} with the suitable sign, it is straightforward to see
that the three generators of $I_{E(T)}$ satisfy (\ref{Laguerre di Coolidge}). Therefore eigenschemes of tensors arise as base loci of Laguerre nets.
\end{remark}

We recall that a net of degree $d\ge 4$ is called {\it irreducible} if all the curves of $\Lambda$ are irreducible. As shown in \cite[Theorem 7.3.13]{Dolgachev}, an irreducible Laguerre net is characterized by one of the following equivalent conditions.

\begin{theorem}\label{Dolg}
 Let $d\ge 4$ and let $V\subseteq \C[x_0,x_1,x_2]_d$ be a vector subspace of dimension $3$. Assume that $\Lambda=\p(V)$ is an irreducible net. Then the following properties are equivalent.
\begin{enumerate}[(i)]
  \item\label{item: i} There exists a basis $f_0, f_1, f_2$ of $V$ such that
$$
x_0f_0 + x_1f_1 + x_2f_2 = 0.
$$
\item\label{item: ii} For any basis $f_0, f_1, f_2$ of $V$, there exist three linearly independent
linear forms $l_0, l_1, l_2$ such that
$$
l_0f_0 + l_1f_1 + l_2f_2 = 0.
$$
\item \label{item: Laguerre} There exists a basis $f_0, f_1, f_2$ of $V$ such that
$$
f_0 =x_1g_2 - x_2g_1, \quad f_1 =x_2g_0-  x_0g_2,\quad f_2 =x_0g_1 - x_1g_0,
$$
where $g_0$, $g_1$, $g_2$ are homogeneous forms of degree $d-1$.
\item\label{item: iv} The base locus of a general pencil in $\Lambda$ is the union of the base locus
of $\Lambda$ and a set of $d-1$ collinear points.
\end{enumerate} 
\end{theorem}

Point \eqref{item: iv} of Theorem \ref{Dolg} gives some 
insight on subpencils of irreducible Laguerre nets, and we shall see that such information allows us to deduce some geometric properties of the rational map associated to an irreducible Laguerre net.

The eigenscheme of the general tensor is zero-dimensional and reduced, hence the Laguerre net $\p(I_{E(T)}(d))$ has no fixed components. However, it does not need to be irreducible. 
As an example, consider the Fermat polynomials $f=x_0^d + x_1^d +x_2^d$; in this case the generators $x_i\de_j f-x_j\de_if$ of $I_{E(f)}$ are reducible. 
Therefore we are going to study  component-free Laguerre nets, rather than irreducible ones.
Let us first fix some notation. Let $Z$ be the reduced zero-dimensional eigenscheme of a tensor $T\in (\C^3)^{\otimes d}$. 
Let us indicate by
$$
\Lambda_Z  = \p (I_Z(d))
$$ 
the net spanned by the generators of $I_Z$. Similarly to what we did in Section \ref{sec: bateman} for $d=3$, we can consider the rational map
\begin{align*}
\lambda_Z : \p^2 &\dasharrow \p (I_Z(d))^\vee\\
P&\longmapsto (f_0(P):f_1(P):f_2(P)),
\end{align*}
where 
$f_0, f_1, f_2$ is the basis (\ref{generatori determinantali}) of $I_Z(d)$. Geometrically, if $P\in \p^2 \setminus Z$, then
\[
    \lambda_Z(P)=\{[g]\in\p(I_Z(d))\mid g(P)=0\}.
\]
In other words,  $\lambda_Z(P)$ is the pencil consisting of the degree $d$ plane curves containing $Z$ and $P$. 
By blowing up $\p^2$ at $Z$ we obtain a generically finite morphism
$$
\widetilde \lambda _Z : \Bl_Z \p^2 \to \p^2.
$$
\begin{definition}
The map $\widetilde \lambda _Z : {\rm Bl}_Z \p^2 \to \p^2$, which resolves the indeterminacy locus of $\lambda_Z$, is called the \emph {Laguerre morphism} associated 
to $Z$.
\end{definition}
 If $\Lambda_Z$ is an irreducible net, from Theorem \ref{Dolg}\eqref{item: iv} we immediately
get that the morphism $\widetilde \lambda _Z$ is generically finite of degree $d-1$, and every finite fiber of the rational map $\lambda_Z$ consists of $d-1$ collinear, not necessarily distinct, points. The following result shows that an analogous statement holds for Laguerre nets which are not necessarily irreducible.

\begin{theorem}\label{thm: facile}
Let $d\ge 4$. Let $T\in (\C^3)^{\otimes d}$ such that $E(T)$ is reduced of dimension 0 and set $Z=E(T)$. Then the Laguerre map $\lambda_Z$ is generically finite of degree $d-1$ and its finite fibers consist of zero-dimensional subschemes of degree $d-1$ contained in a line. 
Moreover, every curve contracted by $\lambda_Z$ is a line and the number $\nu$ of contracted lines satisfies $\nu \le 3(d-1)$.

As a consequence, if $k\in\{2,\dots,d-1\}$ then no $kd$ points of $E(T)$ lie on a curve of degree $k$.
\end{theorem}

\begin{proof}
By Remark \ref{non è restrittivo considerare i parzialmente simmetrici}, we assume that $T=(g_0,g_1,g_2)$ is partially symmetric and we consider generators $f_0,f_1,f_2$ of $I_{Z}$ of the form (\ref{generatori determinantali}). 

A direct computation shows that for any point $P=(P_0:P_1:P_2)\in \p^2 \setminus Z$, the point $\lambda_Z(P)=(f_0(P):f_1(P):f_2(P))$ is the intersection of the two lines
$$
P_0 x_0 + P_1 x_1+ P_2 x_2  = 0\mbox{ and }g_0(P) x_0 +
g_1(P) x_1+ g_2(P) x_2 = 0.
$$
So for the general $Q=(Q_0:Q_1:Q_2)\in \p^2$, the fiber $\lambda_Z^{-1}(Q)$ consists of the points $P\in \p^2$ such that
\begin{equation}\label{polars}
P_0 Q_0 + P_1 Q_1+ P_2 Q_2  =  g_0(P) Q_0 +
g_1(P) Q_1+ g_2(P) Q_2 = 0,
\end{equation}
that is the general fiber is the intersection of the polar line $L_Q$ relative to the isotropic conic and the curve of equation $Q_0 g_0  +
Q_1 g_1 + Q_2 g_2=0$. 
Hence, on the blow-up $\Bl_Z \p^2 \subseteq \p^2 \times \p^2$, the fibers of $\widetilde \lambda_Z$ are generically
contained in the divisor $W$ with bihomogeneous equation $x_0 y_0 + x_1 y_1 + x_2 y_2=0$. Since both $\Bl_Z \p^2$ and $W$ are irreducible, and since $\widetilde \lambda_Z$ is the restriction of the second projection
$p_2 : \p^2 \times \p^2 \to \p^2$, it follows that $\Bl_Z \p^2\subseteq W$; in particular, every fiber of $\widetilde \lambda_Z$ is contained in a line, and by construction the same holds for $\lambda_Z$. As a consequence, the map $\lambda_Z$ contracts only lines.

To bound the number of contracted lines, we observe that the ramification divisor of $\lambda_Z$ 
is exactly the {\it Jacobian divisor} $J$ of the net $\Lambda_Z$, given by the determinant of the matrix with entries the partial derivatives of $f_0,f_1,f_2$. Indeed, a point  $P\in \p^2\setminus Z$ is a ramification point if and only if the pencil of degree $d$ curves with base locus $Z \cup P$ consists of curves intersecting tangentially at $P$, and there is always a singular curve in the pencil with $P$ as a singular point
(see \cite[Book I, Chapter IX, Theorem 25]{Cool}). The locus of all singular points of the curves in the net is given precisely by the Jacobian divisor. Such a divisor has degree $3(d-1)$. We claim that $J$ contains every contracted line $L$; indeed, since $\lambda_Z (L)$
is a point for any $P\in L \setminus Z$, the pencil $\Gamma_P$ through $Z\cup P$
is constant, hence all the points of $L$ are fixed for $\Gamma_P$. This means that $L$ is a fixed component of $\Gamma_P$. We write $\Gamma_P= L +|C_P|$, where $|C_P|$ is a pencil of degree
$d-1$ curves. For any point $R\in L\setminus (Z\cup P)$, there exists a curve of $|C_P|$ passing through $R$, hence the general point of $L$
is a singular point for some curve of $\Gamma_P$, and hence for some curve of $\Lambda_Z$. As every contracted curve is a line, the maximal number of such lines is $\deg J=3(d-1)$.

Finally, if $kd$ points of $Z$ belong to a curve $C$ of degree $k$, then for every $R\in C$, the pencil through
$Z \cup R$ has $C$ as a fixed component; it follows that $C$ is contracted by $\lambda_Z$. By the first part of the present proof, the only possibility is $k=1$.
\end{proof}

\begin{remark}
The bound $\nu \le 3(d-1)$ is sharp; indeed, the Fermat polynomials
$$
f=x_0^d + x_1^d +x_2^d
$$
have a reduced, zero-dimensional eigenscheme of degree $\frac{(d-1)^3-1}{d-2}$. They have exactly $3(d-1)$ sets of $d$ 
collinear points, and every such line is contracted by $\lambda_{E(f)}$.  
The ramification divisor of $\lambda_{E(f)}$ splits as the union of $3(d-1)$ lines, and 
the branch locus of $\widetilde \lambda_{E(f)}$ is 
the union of $3(d-1)$ distinct points. For $d=4$, the eigenscheme is the set
\begin{align*}
\{& (1:0:0),(0:1:0),(0:0:1),\\
    &(1:1:0),(1:0:1),(0:1:1),(1:-1:0),(1:0:-1),(0:1:-1),\\
    &(1:1:-1),(1:-1:1),(-1:1:1),(1:1:1)\}
\end{align*}
If we label these points $p_0,\dots,p_{12}$, then the 9 quadruples of collinear points are
\begin{align*}
    \{0, 1, 3, 6\},
 \{0, 2, 4, 7\},
 \{0, 5, 11, 12\},
 \{0, 8, 9, 10\},
 \{1, 2, 5, 8\}\\
 \{1, 4, 10, 12\},
 \{1, 7, 9, 11\},
 \{2, 3, 9, 12\},
 \{2, 6, 10, 11\}.
\end{align*}
\end{remark}

By Proposition \ref{prop: configurations}, zero-dimensional eigenschemes of tensors in $(\C^{n+1})^{\otimes d}$ never contain $d+1$ collinear points. Now we focus on tensors with zero-dimensional reduced eigenscheme containing $d$ collinear points. The next result can be seen as a generalization of Proposition \ref{pro: plane cubics, general position}.
 
 \begin{theorem}\label{thm : general no d collinear}
Let $d\ge 3$. The eigenscheme $E(f)$ of a general homogeneous polynomial $f\in\C[x_0,x_1,x_2]_d$ contains no $d$ collinear points. As a consequence, the general element of $(\C^3)^{\otimes d}$ has no $d$ collinear eigenpoints. \end{theorem}
\begin{proof}
Consider a symmetric tensor $f\in\C[x_0,x_1,x_2]_d$. The case $d=3$ is solved in Proposition \ref{pro: plane cubics, general position}, so we assume that $d\ge 4$ and we consider the Laguerre map associated to $E(f)$. Since any subset of $d$ collinear points determines a line contracted by $\lambda_{E(f)}$, we are interested in characterizing the locus of polynomials which give rise to such a contraction.

By the proof of Theorem \ref{thm: facile}, the fiber of $\lambda_{E(f)}$ over $Q=(Q_0:Q_1:Q_2)$ is
given by 
$V(Q_0 x_0 + Q_1 x_1+ Q_2 x_2,\ Q_0\partial_0 f   +Q_1\partial_1 f + Q_2\partial_2 f)$, hence
we see that a fiber is a line if and only if either the
first polar $P_Q (f)=Q_0\partial_0 f  +Q_1\partial_1 f + Q_2\partial_2 f$ is identically zero, or the polar line $L_Q=Q_0 x_0 + Q_1 x_1+ Q_2 x_2$ is a component of $P_Q (f)$.
The first case occurs if and only if $V(f)$ consists of concurrent lines (see \cite[Book I, Chapter 9, Theorem 30]{Cool}). To analyse the second case, we
write $L_Q$ in parametric equations with parameters $(t_0:t_1)$. By substituting such expressions in $P_Q(f)$, we get a
homogeneous polynomial of degree $d-1$ in $t_0, t_1$. Therefore the condition $V(L_Q) \subseteq V(P_Q(f))$ is satisfied if and only if the $d$ coefficients of such a polynomial are zero. These equations are bihomogeneous of bidegree $(d,1)$ in the $Q_i$ and the coefficients of $f$, so they determine $d$ hypersurfaces in $\p^2 \times 
\p(\C[x_0,x_1,x_2]_d)$. Let us set
${\mathcal L} \subseteq \p^2 \times 
\p(\C[x_0,x_1,x_2]_d)$ to be their zero  locus. Let $\pi_1$ and $\pi_2$ denote the projections. We want to prove that $\pi_2({\mathcal L})\neq\p(\C[x_0,x_1,x_2]_d)$. We claim that the $d$ equations defining $\p(\C[x_0,x_1,x_2]_d)$ are all independent, so that
\begin{equation}\label{codim}
\codim\mathcal{L} = d.
\end{equation}
Indeed, the restriction of the first projection $\pi_1 : {\mathcal L} \to \p^2$ is surjective, and all its fibers have codimension at most $d$. Since the codimension is 
upper semicontinuous, in order to prove \eqref{codim} it suffices to
exhibit a specific fiber having codimension $d$. Let $Q=(0:0:1)$. Then 
$$
L_Q=(t_0:t_1:0)\mbox{ and }P_Q(f)= \partial _2 f, 
$$
so the condition $V(L_Q) \subseteq V(P_Q(f))$ becomes 
$
\partial_2 f (t_0,t_1,0)\equiv 0.
$
This is equivalent to
require that all the coefficients 
of the monomials 
$
x_0 ^{d-1} x_2,
x_0^{d-2}x_1 x_2, \dots, x_1 ^{d-1}x_2
$
are zero, and the latter are $d$ linearly independent conditions. Hence $\codim\pi_1^{-1}(Q)=d$ and (\ref{codim}) follows. Therefore
$$
\codim \pi_2({\mathcal L}) \ge d-2>0
$$
and our statement follows.

Now we consider tensors that are not necessarily symmetric. Let ${\mathcal C}\subseteq (\p^2)^{(d^2-d+1)}$ be the locus 
of all degree $d^2-d+1$ reduced subschemes 
containing $d$ collinear points. Notice that $\dim\ov{{\mathcal C}}=2+d +2(d-1)^2 < \dim (\p^2)^{(d^2-d+1)}$, so $\overline {\mathcal C}$ is a proper closed subscheme of $(\p^2)^{(d^2-d+1)}$. In order to conclude, we just need to show that there is a symmetric tensor $f$ such that $E(f)\in(\p^2)^{(d^2-d+1)}\setminus\ov{{\mathcal C}}$. By the first part of this theorem, it is enough to take a general polynomial.
\end{proof}

\begin{remark}
Now we want to highlight a connection with some existing literature. We identify $\p (I_Z(d))^\vee =\p^2$. Let us determine the class of $S=\Bl_Z \p^2$ in the Chow ring $A(\p^2 \times \p^2)$. By choosing $L_1$ and $L_2$ as generators of the Picard groups of the two factors, and by setting $p_i:\p^2 \times \p^2 \to \p^2$ to be the two projections, we have that the two divisors
$h_1=p_1^\star L_1$ and $h_2=p_2^\star L_2$
are generators for $A(\p^2 \times \p^2)$. Then it is simple to check that
the class of $S$ in $A (\p^2 \times \p^2)$ is given by
$$
[S]=(d-1)h_1^2
+ dh_1 h_2 + h_2^2.
$$

 As seen in the proof of Theorem \ref{thm: facile}, an eigenscheme $Z$ determines an embedding of $S$ in the divisor $W\subseteq\p^2\times\p^2$ defined by the equation $x_0 y_0 + x_1 y_1+x_2 y_2=0$. The class of $W$ is $h_1 +h_2$, and the restriction of the second projection $p_{2|W}$ realizes $W$ as a projective bundle over $\p^2$. Observe that such an equation corresponds to
the projective bundle associated to the twisted universal quotient bundle ${\mathcal Q}(1)$ of $\p^2$,
which is in turn isomorphic to
the tangent bundle ${\mathcal T}_{\p^2}$, so that $W=\p({\mathcal T}_{\p^2})$.

Moreover, we have seen in the proof of Theorem \ref{thm: facile} that $S$ is contained in the intersection of the two divisors determined by the equation (\ref{polars}), that is, $S\subseteq W \cap D'$ where $D'$ has equation 
\begin{equation}\label{eq_D'}
 y_0 g_0(x_0,x_1,x_2)+ y_1 g_1(x_0,x_1,x_2)+y_2 g_2(x_0,x_1,x_2)=0.
 \end{equation}
The class of such a complete intersection is $$(h_1 + h_2) \cdot ((d-1)h_1+h_2)= (d-1)h_1^2+ dh_1 h_2 + h_2^2=[S],$$
hence $S=W\cap D'$. By setting $t_1= {h_1}_{|W}$ and $t_2= {h_2}_{|W}$, we see that $S$ determines uniquely a section in $H^0 (\oo_W ((d-1)t_1 + t_2))$.
 
Finally, we believe that our construction is related to that performed in \cite[Section 3.1]{OO} to determine the degree of eigenschemes, and in \cite[Section 5.2]{Abo} to study the dimension and degree of eigendiscriminants. According to
\cite[Lemma 5.6]{Abo}, given a vector space $V$ and an order $d$ tensor $A \in V^{\otimes d}$,
there exists a unique section class $[s] \in \p (H^0({\mathcal T}_{\p(V)}(d-2)))$
such that the zero scheme $(s)_0$ is equal to the eigenscheme of $A$. Then we notice that
$$
p_{1_\star} \oo_{ W} ((d-1)t_1 + t_2) \cong {\mathcal T}_{\p^2} (d-2).
$$
\end{remark}

\begin{remark}
Let us determine the branch divisor $\widetilde \lambda_{Z\star}  K_{\widetilde \lambda_Z}$ for $d=4$. In this case the Laguerre morphism is generically finite of degree 3. Observe that
$\widetilde \lambda_{Z\star}  p_1^\star L_1 \sim  4 L_2$; indeed, 
by construction we have ${\widetilde \lambda}_{Z\star} { p_1^\star} L_1
=\lambda_{Z\star} L_1$, and the image 
of a general line under the Laguerre map
is a quartic curve, since $\lambda_Z$ is defined by degree $4$ polynomials. Moreover, since the exceptional divisors $E_i$ on $S$ are lines for any $i\in\{1,\dots, 13\}$, we have $\widetilde \lambda_{Z\star} E_i \sim  L_2$. Finally, from the definition of the push-forward cycle and from the fact that $\widetilde \lambda_Z$ is a degree three cover, we have
 $\widetilde {\lambda}_{Z\star} {\widetilde \lambda_Z}^\star L_2 \sim  3 L_2$.
 
This implies that
$\widetilde \lambda_{Z\star} K_{\widetilde \lambda_Z} \sim 10 L_2$.
When $\widetilde\lambda_Z$ is a finite cover,
by \cite[Proposition 8.1]{Miranda} the dual of
its {\it Tschirnhausen bundle} is precisely ${\mathcal T}_{\p^2}(1)$. 
Moreover, by \cite[Proposition 4.1]{Miranda}, 
$
\widetilde\lambda_{Z\star} K_{\widetilde \lambda_Z} \sim 2 c_1 ({\mathcal T}_{\p^2}(1)),
$
and indeed we have
$
c_1 ({\mathcal T}_{\p^2}(1))= c_1( {\mathcal T}_{\p^2}) + 2L_2 =5 L_2.
$
On the other hand, by \cite[Lemma 10.1]{Miranda},
the branch divisor of a general triple cover $\rho$
between smooth surfaces has only cusps as singularities, corresponding to the total ramification points, and their number is $3 c_2 ({\mathcal E})$, where ${\mathcal E}$ is the Tschirnhausen bundle of $\rho$. Then the branch divisor has $3c_2(\T_{\p^2}(1))=3\cdot 7=21$ cusps. This setting is confirmed by the following example.
\end{remark}
\begin{example}\label{example: maple2}
Consider the irreducible form $f = x_0^3x_2-x_0x_1^2x_2-x_0x_1x_2^2+x_1^4+x_2^4$. The scheme $E(f)$ is reduced, zero-dimensional and has degree $13$. The Jacobian divisor is the curve defined by the polynomial
\scriptsize{
\begin{align*}
& 20x_0^8x_1-44x_0^6x_1^3-32x_0^5x_1^4-164x_0^4x_1^5+32x_0^3x_1^6+60x_0^2x_1^7+8x_0^8x_2-64x_0^6x_1^2x_2+224x_0^5x_1^3x_2\\
&+72x_0^4x_1^4x_2+96x_0^3x_1^5x_2+240x_0^2x_1^6x_2-124x_0^6x_1x_2^2+48x_0^5x_1^2x_2^2+220x_0^4x_1^3x_2^2-48x_0^3x_1^4x_2^2+348x_0^2x_1^5x_2^2\\&-60x_1^7x_2^2-44x_0^6x_2^3-80x_0^5x_1x_2^3-56x_0^4x_1^2x_2^3-608x_0^3x_1^3x_2^3-556x_0^2x_1^4x_2^3-720x_0x_1^5x_2^3-112x_1^6x_2^3\\&+16x_0^5x_2^4+108x_0^4x_1x_2^4+144x_0^3x_1^2x_2^4+328x_0^2x_1^3x_2^4+96x_0x_1^4x_2^4-52x_1^5x_2^4+52x_0^4x_2^5+320x_0^3x_1x_2^5\\&+216x_0^2x_1^2x_2^5+720x_0x_1^3x_2^5+52x_1^4x_2^5-32x_0^3x_2^6-624x_0^2x_1x_2^6-96x_0x_1^2x_2^6+112x_1^3x_2^6-60x_0^2x_2^7+60x_1^2x_2^7.\end{align*}}\normalsize
The branch locus is a curve of degree $10$. The curve $B(f)$ has equation
\scriptsize{\begin{align*}&4x_0^{10}+24x_0^9x_1-408x_0^8x_1^2-76x_0^7x_1^3+768x_0^6x_1^4+96x_0^5x_1^5-356x_0^4x_1^6-24x_0^3x_1^7-24x_0^2x_1^8-4x_0x_1^9+144x_0^8x_1x_2\\    &+288x_0^7x_1^2x_2-72x_0^6x_1^3x_2-720x_0^5x_1^4x_2-720x_0^4x_1^5x_2+720x_0^3x_1^6x_2+360x_0^2x_1^7x_2-20x_0^8x_2^2-84x_0^7x_1x_2^2\\    &+3420x_0^6x_1^2x_2^2+2856x_0^5x_1^3x_2^2-5027x_0^4x_1^4x_2^2-1284x_0^3x_1^5x_2^2-90x_0^2x_1^6x_2^2+96x_0x_1^7x_2^2+25x_1^8x_2^2-16x_0^7x_2^3\\    &-568x_0^6x_1x_2^3+3216x_0^5x_1^2x_2^3+1448x_0^4x_1^3x_2^3+4600x_0^3x_1^4x_2^3+672x_0^2x_1^5x_2^3-2024x_0x_1^6x_2^3+16x_1^7x_2^3+36x_0^6x_2^4\\    &-488x_0^5x_1x_2^4-5166x_0^4x_1^2x_2^4-4416x_0^3x_1^3x_2^4+11730x_0^2x_1^4x_2^4-1160x_0x_1^5x_2^4-60x_1^6x_2^4+48x_0^5x_2^5+552x_0^4x_1x_2^5\\    &-14424x_0^3x_1^2x_2^5+1752x_0^2x_1^3x_2^5+1536x_0x_1^4x_2^5-48x_1^5x_2^5-27x_0^4x_2^6+828x_0^3x_1x_2^6-9878x_0^2x_1^2x_2^6+1600x_0x_1^3x_2^6\\    &+46x_1^4x_2^6-48x_0^3x_2^7-96x_0^2x_1x_2^7-552x_0x_1^2x_2^7+48x_1^3x_2^7+6x_0^2x_2^8-276x_0x_1x_2^8-12x_1^2x_2^8+16x_0x_2^9-16x_1x_2^9+x_2^{10}.\end{align*}}\normalsize
By means of the computer algebra system Macaulay2 we verify that the singular locus of $B(f)$ is a zero-dimensional scheme of degree $42$. The radical of its defining ideal has degree $21$. Since ordinary nodes appear with multiplicity one in the Jacobian scheme, while ordinary cusps have multiplicity two, and other singularities have higher multiplicity, we believe that the branch locus has indeed $21$ cusps as singularities.
\end{example}

\section{Characterization of configurations of eigenpoints}\label{sec: seven}
In this section we shall prove the converse of Theorem \ref{thm: facile}. We shall make use of the {\it numerical character} defined 
in \cite[Definition 2.4]{GrusonPeskine}. The numerical character is a sequence of positive integers associated to a zero-dimensional scheme $Z\subseteq\p^2$ which allows to read off geometric properties from the minimal free resolution of $I_Z$. 
\begin{definition} \label{def: char num}
    Let $Z \subseteq \p^2$ be a zero-dimensional subscheme and set
    $$
    d= \min\{ k\in \Z\mid\dim I_Z(k)\neq 0\}.
    $$
    Set $R=\C[x_0,x_1]$ and let $\C[Z]=\C[x_0,x_1,x_2]/I_{Z}$  be the homogeneous coordinate ring
    of $Z$. Assume that the line $x_2=0$ is in general position with respect to $Z$. Then the natural morphism $R \to \C[Z]$
    gives $\C[Z]$ a structure of 
    graded
    $R$-module of finite type.  A minimal resolution of $\C[Z]$ as an
    $R$-module has the form
    $$
    0 \to \bigoplus _{i=0}^{d-1} R(-n_i)
    \to \bigoplus _{i=0}^{d-1} R(-i) \to \C[Z] \to 0.
    $$
    The integers $n_i$
    are called
    he {\it numerical character}  of $Z$. Such a sequence is usually denoted by  $$\chi(Z)= (n_0,\dots,n_{d-1}),$$ where 
    $n_0 \ge n_1 \ge \ldots  \ge n_{d-1} \ge d$. Finally, if $n_i\le n_{i+1}+1$ for any $i\in\{0,\dots,d-2\}$, we will say that $\chi(Z)$ is {\emph connected}.
\end{definition}
We will need the following relations. For references, see \cite[Theorem 2.7]{GrusonPeskine} and \cite[Proposition I.2.1]{ElliaStrano}.

\begin{lemma}\label{lemma: proprietà del carattere numerico} Let $Z\subseteq\p^2$ be a zero-dimensional subscheme and let $\chi(Z)=(n_0,\dots,n_{d-1})$ be its numerical character.
\begin{enumerate}
    \item \label{item: grado} The degree of $Z$ is $$\deg Z= \sum _{i=0}^{d-1} (n_i -i).$$
\item \label{item: resolution}\label{gradi_sizigie} Let $S=\C[x_0,x_1,x_2]$. Let
$$    0 \to \bigoplus_{i=1}^{t-1} S(-\alpha_i) \to    \bigoplus_{j=1}^{t} S(-\beta_j)\to I_Z \to 0
    $$
be the minimal free resolution of $I_Z$, where $d<\alpha_1 \le \ldots \le \alpha_{t-1}$ and
    $\beta_1 \le \ldots \le \beta_t$. Then
    \begin{equation*}
  \#\{i \mid \beta_i = d\}  =    \# \{i \mid n_i = d\}+1.
        \end{equation*}
Moreover, for any $s \ge d+1$ we have
\begin{equation*}
\#\{i \mid \alpha_i = s\}=\#\{i \mid \beta_i = s\}-\# \{i \mid n_i = s\}+\# \{i \mid n_i = s-1\}.
\end{equation*}
    \end{enumerate}
\end{lemma}

We also recall a result from \cite[Page 112]{ElliaPeskine}:
\begin{proposition}\label{ellia}
Let $Z\subseteq \p^2$ be a zero-dimensional subscheme and let $\chi(Z)=(n_0,\dots,n_{d-1})$ be its numerical character. If $n_{s-1}> n_s +1$ for some $1\le s\le d-1$, then there exists
a curve $C$ of degree $s$ such that 
$$
\chi (C\cap Z)= (n_0, \dots, n_{s-1})\mbox{ and }
\chi({\rm Res}_C(Z))=(n_s -s, \dots, n_{d-1}-s).
$$
\end{proposition}

These results allow us to prove the following theorem, generalizing \cite[Theorem 5.1]{ASS}. Observe that we can assume that the line $x_2=0$ is in general position with respect to $Z$, up to an orthogonal transformation.

\begin{theorem}\label{teorema finale}
Let $d\ge 3$ and let $Z \subseteq \p^2$ be a reduced subscheme of dimension $0$ and degree $d^2-d+1$.
Assume that
\begin{enumerate}
\item \label{item: 3 generatori in grado d}
 $\dim I_Z(d)=3$,
 \item \label{item: no d+1 collineari}no $d+1$ points of $Z$ are collinear, and
    \item \label{item: no kd} for any $k\in\{2, \dots, d-1\}$, no $kd$ points of $Z$ lie on a degree $k$ curve.
\end{enumerate}
Then there exists a tensor $T\in (\Sym^{d-1}\C^{3})^{\oplus 3}$ such that $E(T)=Z$.
\end{theorem}

\begin{proof}
For $k=d-1$, assumption \eqref{item: no kd} ensures that $d$ is the minimum degree of a curve containing $Z$ (compare to Definition \ref{def: char num}). Let $\chi(Z)=(n_0,\dots,n_{d-1})$  be the numerical character of $Z$. Since $\dim I_Z(d)=3$, Lemma \ref{lemma: proprietà del carattere numerico}\eqref{item: resolution} implies that  
$$
n_{d-1}= n_{d-2}=d\mbox{ and } n_{d-3}\ge d+1.
$$
Now we claim that $\chi(Z)$ is connected. Indeed, assume by contradiction that $\chi(Z)$ is not connected and let $k\in\{1,\dots,d-2\}$ be the maximal integer such that
$$
n_{k-1} > n_k+1.
$$
By Proposition \ref{ellia}, there exists a curve of degree $k$ containing a subset $Z'$ of $Z$, with
$$\chi(Z^\prime)=(n_0,\dots,n_{k-1}).
$$
In order to find a contradiction, we are going to show that $\deg(Z')\ge kd$. Since
\begin{align*}
\sum _{i=0}^{d-2} n_i&= \sum _{i=0}^{d-1} n_i - 
n_{d-1} = 
\deg (Z) +\sum _{i=0}^{d-1} i -
d\\
&=d^2-d+1 + \frac{d(d-1)}{2} - d=
\frac{(d-1)(3d-2)}{2},
\end{align*}
 we can write
$$
\sum _{i=0}^{k-1} n_i =\frac{(d-1)(3d-2)}{2}-\sum _{i=k}^{d-2} n_i.
$$
If $i\in\{k,\dots,d-3\}$, then the maximality of $k$ implies that 
$n_i \le n_{i+1} +1$. Hence $n_i\leq n_{d-2}+(d-2-i)=2d-2-i$. Notice that the same inequality also holds for $i=d-2$. Thus
$$
\sum_{i=k}^{d-2} n_i \le (d-k-1)(2d-2) -\frac{(d-2)(d-1)}{2}+\frac{k(k-1)}{2}.
$$
Therefore
\begin{align*}
\sum _{i=0}^{k-1} n_i &\ge \frac{(d-1)(3d-2)}{2} -\left((d-k-1)(2d-2) -\frac{(d-2)(d-1)}{2}+\frac{k(k-1)}{2}\right)\\
&
=2k(d-1)-\frac{k(k-1)}{2}.
\end{align*}
It follows that
\begin{equation}\label{diseq}
    \deg (Z') =\sum _{i=0}^{k-1} (n_i-i)
\ge
2k(d-1)-\frac{2k(k-1)}{2}=2k(d-1)-k(k-1).
\end{equation}
So $Z^\prime$ is a subset of at least $2k(d-1)-k(k-1)$ points of $Z$ contained in a degree $k$ curve for some $k\in\{1,\dots,d-2\}$. If $k=1$, then we have $2d-2\ge d+1$ collinear points, in contradiction to assumption \eqref{item: no d+1 collineari}. If $k\ge 2$, then we have at least $kd$ points on a degree $k$ curve, in contradiction to assumption \eqref{item: no kd}. We deduce that $\chi(Z)$ is connected. In particular, $n_{d-3}=d+1$.

Next we show that
$\chi(Z)$ is strictly decreasing for
$i\in\{0,\dots,d-2\}$. Indeed, the lexicographically maximal connected numerical character with $n_{d-1}=
n_{d-2}=d$ and $n_{d-3}=d+1$ is
$$
\chi_{\rm max}=(2d-2,2d-3,2d-4,\dots,d+2,d+1,d,d),
$$
By Lemma \ref{lemma: proprietà del carattere numerico}\eqref{item: grado}, a set of points with this character would have degree
$$
d+\sum_{i=d}^{2d-2} i - \sum_{i=0}^{d-1}i = d+\dfrac{(2d-2)(2d-1)}{2}-d(d-1)= d^2-d+1=\deg Z,$$
thus $\chi(Z)=\chi_{\rm max}$. In particular, there is exactly one integer
among the $n_i$ 
equal to $d+1$, and precisely $n_{d-3}=d+1$. Therefore we can apply Lemma \ref{lemma: proprietà del carattere numerico}\eqref{item: resolution} to get
\begin{align*}
    \#\{i\mid  \alpha_i = d+1\}&=\#\{i \mid  \beta_i = d+1\}-\# \{i \mid  n_i = d+1\}+\# \{i \mid  n_i = d\}\\
&=\#\{i \mid\beta_i = d+1\}
-1+2 \ge 1.
\end{align*}
So there is at least one syzygy in degree $d+1$ among the three generators of degree $d$, of the form
\begin{equation}\label{prima_sizigia}
l_0 f_0 +l_1f_1 +l_2f_2\equiv 0.
\end{equation}
Next we observe that our geometric assumptions on $Z$ imply that the linear forms $l_0,l_1,l_2$ are linearly independent. Indeed, assume by contradiction that $l_2= \alpha l_0 +\beta l_1$ for some $\alpha,\beta\in\C$. By (\ref{prima_sizigia}) we have
$$
l_0 (f_0 + \alpha f_2) \equiv -l_1 (f_1 +\beta f_2),
$$
which implies that there exists a degree $d-1$ form $h$ such that
$$
f_0 +\alpha f_2 = l_1 h\mbox{ and } f_1 +\beta f_2 = -l_0 h.
$$
These relations imply that $I_{Z}$ is generated also by
$$
I_{Z}=(f_0,f_1,f_2)=(f_2, l_0 h, l_1 h).
$$
In particular, $Z$ contains $V(f_2,h)$, which are $(d-1)d$ points on the degree $d-1$ curve $V(h)$. But this contradicts our hypothesis (\ref{item: no kd}) with $k=d-1$. By writing explicitly the forms $l_0,l_1,l_2$, we can rewrite the relation
(\ref{prima_sizigia}) in the form
\begin{equation*}\label{nuova sizigia}
x_0 \tilde f_0 + x_1 \tilde f_1 +x_2 \tilde f_2=0,
\end{equation*}
where $\tilde f_0, \tilde f_1, \tilde f_2$ are three other generators of $I_{Z}$. By Lemma \ref{lem: i implica iii} we conclude that $\tilde f_0, \tilde f_1, \tilde f_2$ are the determinantal equation of the eigenscheme of a partially symmetric tensor $T$.
\end{proof}

In Theorem \ref{teorema finale}, we assumed that $Z$ is reduced. Actually, the same argument proves the analogous result for non-reduced zero-dimensional subschemes of $\p^2$.

\section*{Acknowledgments}
The authors are grateful to Matteo Gallet and Alessandro Logar for useful conversations, to Anna Seigal, Luca Sodomaco and Giorgio Ottaviani for their helpful comments, to Hirotachi Abo for having informed them about the preprint \cite{Abo} and to Aldo Conca for bringing Remark \ref{rem: saturate} to their attention. Finally, we thank Bernd Sturmfels for his encouragement and his positive feedback on the preliminary version of this work.

\bibliographystyle{plain}
\bibliography{ooms}

\end{document}